\documentclass[11pt]{article}

\usepackage{latexsym}
\usepackage{amssymb}
\usepackage{amsthm}
\usepackage{amscd}
\usepackage{amsmath}
\usepackage{mathrsfs}
\usepackage[all]{xy}
\usepackage{hyperref} 
\usepackage[usenames,dvipsnames]{color}
\usepackage{graphicx,eepic}
\usepackage{float}

\usepackage[small,nohug,heads=vee]{diagrams} 
\diagramstyle[labelstyle=\scriptstyle]

\usepackage{color}

\theoremstyle{definition}
\newtheorem* {theorem*}{Theorem}
\newtheorem{theorem}{Theorem}[section]

\newtheorem* {problem*}{Problem}
\theoremstyle{definition}

\newtheorem* {example*}{Example}
\newtheorem{fact}[theorem]{Fact}
\newtheorem{lemma}[theorem]{Lemma}
\theoremstyle{definition}
\newtheorem{definition}[theorem]{Definition}
\theoremstyle{definition}
\newtheorem* {notation}{Notation}

\newtheorem{proposition}[theorem]{Proposition}
\newtheorem{corollary}[theorem]{Corollary}

\newtheorem* {remark}{Remark}
\theoremstyle{definition}
\newtheorem {example}[theorem]{Example}
\theoremstyle{definition}

\theoremstyle{definition}

\theoremstyle{definition}

\xyoption{dvips}

\usepackage{fullpage}

\numberwithin{equation}{section}

\def\qquand{\qquad\text{and}\qquad}
\def\quand{\quad\text{and}\quad}

\def\({\left(}
\def\){\right)}
   \newcommand{\FF}{\mathbb{F}}       \newcommand{\cP}{\mathcal{P}} 
 \newcommand{\cK}{\mathcal{K}} \newcommand{\cO}{\mathcal{O}} 
  \newcommand{\cS}{\mathcal{S}} 
 
\newcommand{\cC}{\mathcal{C}}

\def\NN{\mathbb{N}}
    \def\ZZ{\mathbb{Z}}   \def\H{\mathcal{H}} \def\Ind{\mathrm{Ind}}    \def\Res{\mathrm{Res}}       
  \def\wt{\widetilde}
 \def\L{\mathcal{L}} \def\sh{\mathrm{sh}}

   \newcommand{\Tr}{\mathrm{Tr}}
\newcommand{\h}{\mathfrak{h}}
\newcommand{\fn}

\newcommand{\One}{{1\hspace{-.13cm} 1}}

\newcommand{\sgn}{\mathrm{sgn}}

\def\fk{\mathfrak}

\def\barr{\begin{array}}
\def\earr{\end{array}}
\def\ba{\begin{aligned}}
\def\ea{\end{aligned}}
\def\be{\begin{equation}}
\def\ee{\end{equation}}
\def\cS{\mathcal{S}}

\def\ll{\langle\langle}
\def\rr{\rangle\rangle}

\def\hs{\hspace{0.5mm}}

\def\ben{\begin{enumerate}}
\def\een{\end{enumerate}}

\def\omdef{\overset{\mathrm{def}}}

\def\c{\mathbf{c}}

\newcommand{\tthree}[6]{\xy<0.0cm,0.25cm> \xymatrix@R=0.5cm@C=0.2cm{ *{\bullet} #6 &*{\bullet} #5 &*{\bullet} #4 \\  *{\bullet}  #1 &*{\bullet}  #2 &*{\bullet}  #3  }\endxy}

\newcommand{\tfour}[8]{\xy<0.0cm,0.25cm> \xymatrix@R=0.2cm@C=0.2cm{ *{\bullet} #8 & *{\bullet} #7 &*{\bullet} #6 &*{\bullet} #5 \\  *{\bullet}  #1 &*{\bullet}  #2 &*{\bullet}  #3  &*{\bullet}  #4  }\endxy}

\newcommand{\exone}{\xy<0cm,0cm> \xymatrix@R=0.0cm@C=0.3cm{*{\bullet}1 &*{\bullet}  
}\endxy}

\newcommand{\extwo}[4]{\xy<0cm,0cm> \xymatrix@R=-0.0cm@C=0.3cm{*{\bullet} #1 &*{\bullet} #2 &*{\bullet} #3 &*{\bullet}  #4  
}\endxy}

\newcommand{\exthree}[6]{\xy<0cm,0cm> \xymatrix@R=-0.0cm@C=0.3cm{*{\bullet} #1 &*{\bullet} #2 &*{\bullet} #3 &*{\bullet}  #4 &*{\bullet}  #5 &*{\bullet}  #6
 }\endxy}

\newcommand{\exfour}[8]{\xy<0.0cm,0.0cm> \xymatrix@R=-0.0cm@C=0.3cm{*{\bullet} #1 &*{\bullet} #2 &*{\bullet} #3 &*{\bullet}  #4 &*{\bullet}  #5 &*{\bullet}  #6 &*{\bullet}  #7 &*{\bullet}  #8
}\endxy}

\newcommand{\xone}[1]{\xy<0cm,0cm> \xymatrix@R=-0.0cm@C=0.3cm{*{\bullet} #1 
  }\endxy}

\newcommand{\xtwo}[2]{\xy<0cm,0cm> \xymatrix@R=-0.0cm@C=0.3cm{*{\bullet} #1 &*{\bullet} #2  
  }\endxy}

\newcommand{\xthree}[3]{\xy<0cm,0cm> \xymatrix@R=-0.0cm@C=0.3cm{*{\bullet} #1 &*{\bullet} #2 &*{\bullet} #3 
}\endxy}

\newcommand{\xfour}[4]{\xy<0.0cm,0.0cm> \xymatrix@R=-0.0cm@C=.3cm{*{\bullet} #1 &*{\bullet} #2 &*{\bullet} #3 &*{\bullet}  #4 
}\endxy}

\newcommand{\xfive}[5]{\xy<0.0cm,0.0cm> \xymatrix@R=-0.0cm@C=0.3cm{*{\bullet} #1 &*{\bullet} #2 &*{\bullet} #3 &*{\bullet}  #4  &*{\bullet}  #5 
}\endxy}

\newcommand{\xsix}[6]{\xy<0.0cm,0.0cm> \xymatrix@R=-0.0cm@C=0.3cm{*{\bullet} #1 &*{\bullet} #2 &*{\bullet} #3 &*{\bullet}  #4  &*{\bullet}  #5 &*{\bullet}  #6 
}\endxy}

\newcommand{\xseven}[7]{\xy<0.0cm,0.0cm> \xymatrix@R=-0.0cm@C=0.3cm{*{\bullet} #1 &*{\bullet} #2 &*{\bullet} #3 &*{\bullet}  #4  &*{\bullet}  #5 &*{\bullet}  #6 &*{\bullet}  #7 
}\endxy}

\def\kk{\mathbb{K}}

\def\st{\mathrm{st}}

\def\id{\mathrm{id}}

\def\im{\mathrm{Image}}
\newarrow {Corresponds} <--->

\def\Set{\textbf{Set}}

\def\Sp{\textbf{Sp}}

\def\Fun{\mathrm{Fun}}
\def\Cl{\mathrm{ClassFun}}

\def\GrVec{\textbf{GrVec}}
\def\canon{\mathrm{canon}}

\def\Hopf{\textbf{Hopf}}

\def\ds{\displaystyle}

\def\Vec{\textbf{Vec}}
\def\h{\mathbf{h}}

\def\p{\mathbf{p}}
\def\q{\mathbf{q}}

\def\E{\mathrm{E}}
\def\L{\mathrm{L}}

\def\P{\mathrm{P}}
\def\Q{\mathrm{Q}}

\def\one{\mathbf{1}}

\def\csum{\hs\#\hs}
\def\bfSigma{\mbox{\boldmath$\fk S$}}
\def\bfE{\mbox{\textbf{E}}}

\def\bfL{\mbox{\textbf{L}}}

\def\bfPi{\mbox{\boldmath$\Pi$}}

\def\Sym{\mathbf{Sym}}
\def\NCSym{\mathbf{NCSym}}

\def\S{{\tt S}}
\def\Mobius{\text{M\"obius}}

\def\FB{\mathbf{FB}}

\def\SL{\mathrm{SL}}
\def\USL{\mathrm{USL}}
\def\UO{\mathrm{UO}}
\def\USp{\mathrm{USp}}

\def\SC{{\cS\cC}}

\makeatletter
\renewcommand{\@makefnmark}{\mbox{\textsuperscript{}}}
\makeatother

\allowdisplaybreaks[1]

\UseCrayolaColors

\begin{document}
\title{Strong forms of linearization for Hopf monoids in species
}
\author{Eric Marberg\footnote{This research was conducted with support from the National Science Foundation.}
\\
Department of Mathematics \\
Stanford University \\
{\tt emarberg@stanford.edu}}
\date{}

\maketitle

\begin{abstract}
A vector species is a functor  from the category of finite sets with bijections to vector spaces; informally, one can view this as a sequence of $S_n$-modules. A Hopf monoid (in the category of vector species) consists of a vector species with  unit, counit, product, and coproduct morphisms satisfying several compatibility conditions, analogous to a  graded  Hopf algebra. We say that a Hopf monoid is strongly linearized if it has a ``basis'' preserved by its product and coproduct in a certain sense. We prove several equivalent characterizations of this property, and show that  any strongly linearized Hopf monoid which is commutative and cocommutative possesses four  bases which one can view as analogues of the classical bases of the algebra of symmetric functions. There are natural functors which turn Hopf monoids into graded Hopf algebras,  and applying these functors to strongly linearized Hopf monoids produces several notable families of Hopf algebras. For example, in this way we give a simple unified construction of the   Hopf algebras of superclass functions attached to the maximal unipotent subgroups of three families of classical Chevalley groups.
\end{abstract}

\tableofcontents

\section{Introduction}

Everywhere in this work $\kk$ denotes a field with characteristic zero.
A vector species is a functor  from the category of finite sets with bijections to $\kk$-vector spaces; one can view this as a sequence of $S_n$-modules, one for each natural number $n$. 
A Hopf monoid (in the category of vector species) consists of a vector species with  unit, counit, product, and coproduct morphisms satisfying several compatibility conditions, analogous to a  graded  Hopf algebra.
Background on vector species and Hopf monoids in species
is given in Section \ref{prelim-sect1}.

This work is a sequel to \cite{Me}, which develops the structure theory of Hopf monoids that are self-dual in a certain strong sense. We review these results in Section \ref{prelim-sect2}. Our new results begin in Section \ref{structu-sect} and concern vector species which are \emph{linearized} in the sense that they correspond to  sequences of $S_n$-modules which are permutation representations. For such species there is a notion of \emph{basis}, which one can view as a sequence of $S_n$-sets subject to some conditions; see Section \ref{lin-sect}. From any comonoid in species which is connected, cocommutative, and linearized, there is standard way of constructing a connected Hopf monoid which is  self-dual and linearized. We call the Hopf monoids arising from this construction \emph{strongly linearized}. 
 We show that such Hopf monoids are characterized by a more elementary condition (see Theorem \ref{fl-thm}) and that they are distinguished from Hopf monoids which are merely linearized according to properties of an associated partial order (see Theorem \ref{posetisom-thm}).
 Finally, we show that a Hopf monoid which is commutative, cocommutative, and strongly linearized has four natural bases, which one can view as analogues of the four classical bases of the Hopf algebra of symmetric functions; see Theorem \ref{basis-thm}.

There are several natural functors which one can use to turn Hopf monoids in species into graded Hopf algebras. Applied to strongly linearized Hopf monoids, in particular, these functors 
give rise to  families of Hopf algebras with notable additional structure.
In Section \ref{app-sect} we discuss several examples in which the results in Section \ref{structu-sect} unify constructions at the level of Hopf algebras via this principle.
These applications are broadly divided into three topics, which we introduce briefly as follows.

\subsection{Classical bases of symmetric functions}

 Let $\kk\langle\langle x_1,x_2,\dots\rangle\rangle = \kk\ll x \rr$ be the algebra of formal powers series over $\kk$ in a countably infinite set of noncommuting variables $x_i$. For each positive integer $n$, the symmetric group $S_n$ acts on $\kk\ll x\rr$ by permuting the indices of the variables $x_1,\dots,x_n$, and we write \[\NCSym = \NCSym(x)\] for the subspace of functions of bounded degree in $ \kk\ll x \rr$ which are invariant under all of these actions.
 Wolf's paper \cite{Wolf} is perhaps the first place in the literature to consider this algebra; Rosas and Sagan began its systematic study in \cite{Sagan1}, which has continued in the last decade, e.g., in \cite{AZ,ABT,BergeronZabrocki,preprint,Sagan2}.
There is a natural Hopf algebra structure on this graded algebra (which we review in Section \ref{sym-sect}) and, 
following \cite{BergeronZabrocki,preprint,HNT,Sagan1}, we call $\NCSym$ the \emph{Hopf algebra of symmetric functions in noncommuting variables.}

   $\NCSym$ is one of   several noncommutative analogues of the more familiar and much-studied \emph{Hopf algebra of symmetric functions}, which we denote $\Sym$. Standard references for this object include \cite{Macdonald,Sagan,Stan2} and also the useful notes \cite{ReinerNotes}.
%
%
  The algebra of symmetric functions has four classical bases indexed by integer partitions, each with a simple explicit definition.
   Sagan and Rosas \cite{Sagan1}  show that $\NCSym$ likewise has four simply defined bases, now indexed by set partitions, which are mapped onto (scalar multiples of) the classical bases of $\Sym$ by the  natural homomorphism
 of graded connected Hopf algebras 
 \be\label{ncsym2sym} \NCSym \to \Sym\ee
 which lets the variables $\{x_i\}$ in $\kk\ll x \rr$ commute.
 \emph{A priori}, it is noteworthy that the bases of $\Sym$ should each lift to bases $\NCSym$ in this convenient fashion$-$in particular, in such a way that the Hopf algebraic structures are preserved.
 As an explanation of this phenomenon, we show  
in Section \ref{sym-sect} that these definitions may be unified by a more general construction via our results on strongly linearized Hopf monoids.
   
\subsection{Hopf algebras of superclass functions}\label{sc-intro}

Let $\FF$ be a finite field and write $ \USL_n(\FF)$ for the group of    upper triangular $n\times n$-matrices  
over $\FF$ with all diagonal entries equal to one.
Among other characterizations, this is  the maximal unipotent subgroup of $\SL_n(\FF)$, where we say that a matrix group is \emph{unipotent} if each of its elements is equal to the identity plus some nilpotent matrix.
 The product group $\USL_n(\FF)\times \USL_n(\FF)$ acts on $\USL_n(\FF)$ by  $(g,h) : x \mapsto 1 + g(x-1)h^{-1}$,
 and the \emph{superclasses} of 
 $\USL_n(\FF)$ are the orbits of this action.
We let $\SC(\USL_n,\FF)$ denote the $\kk$-vector space of \emph{superclass functions}, that is, maps $\USL_n(\FF) \to \kk$ which are constant on superclasses.
Such functions are of interest as they have a distinguished basis of characters with many notable properties, and which are the prototypical example of a \emph{supercharacter theory} as defined in \cite{DI}.
Moreover, there is a natural Hopf algebra structure \cite{AZ}
on the graded vector space
 \[\SC(\USL_\bullet,\FF) = \bigoplus_{n\geq 0} \SC(\USL_n,\FF).\]
When $\FF$ is a field with exactly two elements, this Hopf algebra is isomorphic to $\NCSym$.

There are analogous definitions of superclass functions 
on the maximal unipotent subgroups of other finite Chevalley groups.
Given a square matrix $X$, we write $X^\dag$ for the backwards transposed matrix, formed by flipping $X$ about its anti-diagonal.
Assume $\FF$ has an odd number of elements and define
\[ \UO_{2n}(\FF) = \{ g \in \USL_{2n}(\FF) : g^{-1}=g^\dag\}
\quand
\USp_{2n}(\FF) = \{ g \in \USL_{2n}(\FF) : g^{-1} = -\Omega g^\dag \Omega \}\]
where $\Omega$ denotes the $2n\times 2n$ matrix given by 
\[ \Omega = \(\barr{cc} I_n & 0 \\ 0& -I_n \earr\).\]
These are the maximal unipotent subgroups of the  even orthogonal and symplectic groups over $\FF$.
The \emph{superclasses} of these groups are the nonempty sets of the form $\cK \cap \UO_{2n}(\FF)$ and $\cK \cap \USp_{2n}(\FF)$ where $\cK$ is a superclass of $\USL_{2n}(\FF)$.
Andr\'e and Neto \cite{bcd1,bcd2,bcd3} first showed that the spaces of superclass functions on these groups likewise have interesting bases of characters; Andrews's recent preprint \cite{Andrews} gives a uniform construction of these supercharacter theories.

Let $\SC(\UO_{2n},\FF)$ and $ \SC(\USp_{2n},\FF)$ denote  the $\kk$-vector spaces of maps $\UO_{2n}(\FF) \to \kk$ and $\USp_{2n}(\FF) \to \kk$ which are constant on superclasses,
and define the graded vector spaces 
\[ \SC(\UO_\bullet,\FF) = \bigoplus_{n\geq 0} \SC(\UO_{2n},\FF)
\qquand
\SC(\USp_\bullet,\FF) = \bigoplus_{n \geq 0} \SC(\USp_{2n},\FF).\]
Benedetti's work \cite{CB} defines a Hopf algebra structure on $\SC(\UO_\bullet,\FF)$, and there is an analogous structure on $\SC(\USp_\bullet,\FF)$.
In Section \ref{sc-sect} we propose a unified approach to the construction of all three of these Hopf algebras, by showing that they arise from a simple family of strongly linearized Hopf monoids in species; see Theorem \ref{sc-thm}.
This perspective yields an easy proof that the Hopf algebras 
$ \SC(\USL_\bullet,\FF)$ and $ \SC(\UO_\bullet,\FF')$
 considered in \cite{AZ} and \cite{CB} are actually isomorphic whenever $|\FF| = 2|\FF'| - 1$; see Theorem \ref{hiso-thm}.

\subsection{Combinatorial Hopf monoids}
\label{intro-comb-sect}

As introduced in \cite{ABS}, a \emph{combinatorial Hopf algebra}
is a pair $(\H,\zeta)$ where $\H = \bigoplus_{n\geq 0} \H_n$ is a graded connected Hopf algebra over $\kk$ such that each subspace  $ \H_n $ is finite-dimensional, and $\zeta : \H\to \kk$ is an algebra homomorphism (referred to as the \emph{character}). 
A morphism between  combinatorial Hopf algebras $ (\H,\zeta) $ and $ (\H',\zeta')$  is a morphism of graded connected Hopf algebras $\alpha : \H \to \H'$ such that $\zeta = \zeta' \circ \alpha$.
Aguiar, Bergeron, and Sottile \cite[Theorem 4.3]{ABS}
prove that $\Sym$ is a terminal object in the full subcategory of  combinatorial Hopf algebras which are cocommutative.
In Section \ref{comb-sect} we describe how 
this characterization of  $\Sym$ lifts to the  species of set partitions $\bfPi$ (see Example \ref{Pi-ex}).
In detail, we define a category of
 \emph{(cocommutative) combinatorial Hopf monoids} in species and 
 review how results of Aguiar and Mahajan \cite{species} implies the natural Hopf monoid structure on $\bfPi$ is a terminal object in this category; see Theorem \ref{terminal-thm}.
As applications, we use this result to give an alternate construction of the Frobenius characteristic map, identifying class functions on the symmetric group with $\Sym$, and also the definition of a certain  map $\Sym \to \NCSym$ considered in \cite[Section 4]{Sagan1}.

\subsection*{Acknowledgements}

I am grateful to Marcelo Aguiar for answering several questions about Hopf monoids in species, and to Carolina Benedetti for many helpful discussions.

\section{Preliminaries on species} 
\label{prelim-sect1}

The primary reference for this section is Aguiar and Mahajan's book \cite{species}, whose conventions we mostly follow. Additional references for this material  include \cite{species1,species2,blp,joyal}.

\subsection{Vector species}

A \emph{vector species} is a functor $\p : \FB \to \Vec$
where $\FB$ denotes the category of finite sets with bijective maps as morphisms and $\Vec$ denotes the category of $\kk$-vector spaces.
A species $\p$ thus assigns to each finite set $S$ a vector space $\p[S]$ and to 
each bijection  $\sigma : S \to S'$ between finite sets a linear map $\p[\sigma] : \p[S]\to \p[S']$. 
The only condition which these assignments must satisfy is that
\[\label{functorial-eq} \p[\id_S] = \id_{\p[S]} \qquand \p[\sigma \circ \sigma'] = \p[\sigma]\circ \p[\sigma']\]
 whenever $\sigma$ and $\sigma'$ are composable bijections between finite sets.
The composition rule implies that 
$\p[\sigma]$ is always a bijection, since $\p[\sigma^{-1}]$ affords its inverse. 

Vector species form a category in which morphisms are natural transformations. Thus a morphism  $f: \p \to \q$ of vector species  consists of a $\kk$-linear map $f_S : \p[S] \to \q[S]$ for each finite set  $S$, such that 
\[\q[\sigma]\circ f_S = f_{S'} \circ \p[\sigma]\]
for all bijections $\sigma : S \to S'$.
We refer to the map $f_S$ as the \emph{$S$-component} of $f$.

\begin{example}
Two basic vector species are $\one$ and $\bfE$. The former is defined such that $\one[\varnothing] = \kk$ and $\one[S] =0$ for all nonempty sets $S$. The \emph{exponential species} $\bfE$ is given by setting $\bfE[S] = \kk$ for all finite sets $S$ and $\bfE[\sigma] = \id$ for all bijections $\sigma$ between finite sets.
\end{example}

We are most often interested in vector species with the following additional structure.

\begin{definition} 
A \emph{connected species} is a vector species $\p$ with inverse $\kk$-linear bijections $\eta_\varnothing: \kk \to \p[\varnothing]$ and $\varepsilon_\varnothing : \p[\varnothing] \to \kk$.
We refer to 
these maps as the \emph{unit} and \emph{counit} of 
$\p$, respectively. 
\end{definition}

Connected species form a category in which morphisms are natural transformations $f: \p \to \p'$ with the property that 
$f_\varnothing(1_\p) = 1_{\p'}$ where
$1_\p$ and $1_{\p'}$ are the images of $1_\kk \in \kk$ under the units of $\p$ and $\p'$ respectively; see \cite[Fact 2.1.3]{Me}. 

If $\p$ is a vector species with $\p[\varnothing] = \kk$ (for example, $\one$ and $\bfE$),
then,
unless explicitly noted otherwise,
we view $\p$ as the connected species whose unit and counit are the identity maps on $\kk$.
%
Similarly, if $\p[\varnothing]$ is defined as a vector space with a natural basis with one element $1_\p$, then we view $\p$ as the connected species whose unit and counit are the linearizations of the maps $1_\p \mapsto 1_\kk$ and $1_\kk \mapsto 1_\p$.
This convention makes the vector species in each of the following examples connected.

 \begin{example}
 Let $C$ be any set.
Define $\bfE_C$ as the vector species of $C$-valued maps, so that if $S$ is a finite set then
 $\bfE_C[S]$ is the $\kk$-vector space whose basis is the set of  maps $f : S \to C$. If $\sigma : S \to S'$ is a bijection then  $\bfE_C[\sigma]$ acts on the basis of maps by the formula $f\mapsto f\circ \sigma^{-1}$.
\end{example}

Observe that up to isomorphism $\bfE_C$ depends only on the cardinality of $C$, and that $\one \cong \bfE_\varnothing$ and $\bfE \cong \bfE_{\{1\}}$.
Next, recall that a \emph{partition} of a set $S$ is a set of pairwise disjoint nonempty sets (referred to as \emph{blocks}) whose union is $S$. 

\begin{example}
Define $\bfPi$ as the vector species of set partitions, so that $\bfPi[S]$ is the $\kk$-vector space whose basis 
is the set of partitions of $S$. If $\sigma : S \to S'$ is a bijection then $\bfPi[\sigma]$ acts on the basis of set partitions by the formula $X \mapsto \{ \sigma(B) : B \in X\}$.
\end{example}

\begin{example}
Define $\bfL$ as the vector species of linear orders, so that $\bfL[S]$ is the $\kk$-vector space whose basis 
is the set of bijections $\ell: [n] \to S$ (which we refer to as \emph{linear orders} or just \emph{orders}) where $|S| = n$ and $[n] =\{1,\dots,n\}$. If $\sigma : S \to S'$ is a bijection then  $\bfL[\sigma]$ acts on the basis of orders   by the formula $\ell \mapsto \sigma\circ \ell$.
\end{example}

\begin{example}
Define $\bfSigma$ as the vector species of permutations, so that $\bfSigma[S]$ is the $\kk$-vector space whose basis 
is the set of bijections $\pi: S\to S$. If $\sigma : S \to S'$ is a bijection then $\bfSigma[\sigma]$ acts on the basis of permutations   by the formula $\pi \mapsto \sigma \circ \pi \circ \sigma ^{-1}$.
\end{example}

\subsection{Monoidal structures}

Here we review briefly the explicit definitions of  monoids, comonoids, and Hopf monoids in the category of connected species with the \emph{Cauchy product} (see \cite[Definition 8.5]{species}). 
For a more leisurely presentation of this material, see \cite[Sections 8.2 and 8.3]{species} or \cite[Section 3]{Me}.


\begin{notation} In commutative diagrams, whenever we write 
\[
\begin{diagram}
V_1 \otimes V_2 \otimes \dots \otimes V_k && \rTo^{\sim} && V_{\sigma(1)} \otimes V_{\sigma(2)} \otimes \dots \otimes V_{\sigma(k)}
  \end{diagram}
  \]
  where each $V_i$ is a $\kk$-vector
  space and $\sigma \in S_k$ is a permutation, we mean the obvious isomorphism given by linearly extending the map on tensors 
$v_1 \otimes v_2 \otimes \dots \otimes v_k  \mapsto v_{\sigma(1)} \otimes v_{\sigma(2)} \otimes \dots \otimes v_{\sigma(k)}$.
\end{notation}

Suppose $\p$ is a vector species and $\nabla$ and $\Delta$ are systems of $\kk$-linear maps 
\[ \nabla_{S,T} : \p[S]\otimes \p[T]\to\p[S\sqcup T]
\qquand
\Delta_{S,T} : \p[S\sqcup T] \to \p[S]\otimes \p[T]\]
 indexed by pairs of disjoint finite sets $S$, $T$. We say that $\nabla$ is a \emph{product} on $\p$ and that $\Delta$ is a \emph{coproduct} on $\p$ if the diagrams 
 \[\begin{diagram}
 \p[ S \sqcup T] && \rTo^{\p[\sigma]} && \p[S'\sqcup T'] \\
\uTo^{\nabla_{S,T}}&  &&& \uTo_{\nabla_{S',T'}} \\
\p[S] \otimes \p[T] & &\rTo^{ \p[\sigma|_S]\otimes \p[\sigma|_T] }  && \p[S']\otimes \p[T'] 
\end{diagram}
\qquad\qquad
\begin{diagram}
 \p[ S \sqcup T] && \rTo^{\p[\sigma]} && \p[S'\sqcup T'] \\
\dTo^{\Delta_{S,T}}&  &&& \dTo_{\Delta_{S',T'}} \\
\p[S] \otimes \p[T] &&\rTo^{ \p[\sigma|_S]\otimes \p[\sigma|_T] } && \p[S']\otimes \p[T'] 
\end{diagram}
\]
respectively commute
 for any bijection $\sigma : S\sqcup T \to S' \sqcup T'$ with $\sigma(S) = S'$ and $\sigma(T) = T'$.
In turn,
we say that a product $\nabla$ is \emph{associative} (respectively, \emph{commutative}) if the diagram on the left (respectively, right)
\[
\begin{diagram}
  \p[R]\otimes \p[S]\otimes \p[T] &&& \rTo^{\nabla_{R, S} \otimes \id} &&& \p[R\sqcup S]\otimes \p[T] \\
\dTo^{\id\otimes \nabla_{S,T}}&  &&&&& \dTo_{\nabla_{R\sqcup S,T}} \\
\p[R] \otimes \p[S\sqcup T] & &&\rTo^{\nabla_{R,S\sqcup T}}  &&& \p[ R\sqcup S \sqcup T] \end{diagram}
\qquad
\begin{diagram}
  \p[S]\otimes \p[T] && \rTo^{\sim} && \p[T]\otimes \p[S] \\
& \rdTo_{\nabla_{S,T}} & &\ldTo_{\nabla_{T,S}} \\
& & \p[S\sqcup T] 
\end{diagram}
 \]
 commutes for all pairwise disjoint finite sets $R$, $S$, $T$.
A coproduct $\Delta$ is \emph{coassociative} (respectively, \emph{cocommutative}) if 
the same diagrams commute when we replace all arrows labeled by $\nabla$'s with arrows in the reverse direction labeled by $\Delta$'s.

Let $\p$ be a connected species with unit $\eta_\varnothing$ and counit $\varepsilon_\varnothing$;
define $1_\p = \eta_\varnothing(1_\kk)$.
Suppose $\nabla$ is a product on $\p$ and   $\Delta$ is a coproduct on $\p$. We  have the following definitions:
\begin{itemize}
\item The pair $(\p,\nabla)$ is a \emph{connected monoid} if $\nabla$ is associative
and 
\[ \nabla_{\varnothing,S}(1_\p \otimes \lambda) = \nabla_{S,\varnothing}(\lambda \otimes 1_\p) =   \lambda\]
for all $\lambda \in \p[S]$ and all finite sets $S$.
 A morphism  $ (\p,\nabla) \to (\p',\nabla')$ between  connected monoids is a morphism of connected species  $f : \p \to \p'$   which commutes with  products  in the sense that 
$f_{S\sqcup T}\circ \nabla_{S,T} = \nabla_{S,T}' \circ (f_{S} \otimes f_{T})$ for all  disjoint finite sets $S$, $T$.  
 A connected monoid is \emph{commutative} if its product  is commutative.
 
 \item The pair $(\p,\Delta)$ is a \emph{connected comonoid} if $\Delta$ is coassociative
and 
\[ \Delta_{\varnothing,S}(\lambda) = 1_\p \otimes \lambda \qquand \Delta_{S,\varnothing}(\lambda) =   \lambda\otimes 1_\p\]
for all $\lambda \in \p[S]$ and all finite sets $S$.
 A morphism  $ (\p,\Delta) \to (\p',\Delta')$ between  connected comonoids is a morphism of connected species  $f : \p \to \p'$  which commutes with  coproducts in the sense that 
$(f_S\otimes f_T)\circ \Delta_{S,T} = \Delta_{S,T}' \circ f_{S\sqcup T}$ for all  disjoint finite sets $S$, $T$. 
A connected comonoid is \emph{cocommutative} if its coproduct  is cocommutative.

\end{itemize}
Finally, we define a notion of compatibility between product and coproducts.
Namely, we say that a product $\nabla$ and a coproduct $\Delta$ on the same vector species $\p$ are \emph{Hopf compatible} if 
for any
two disjoint decompositions $I = R\sqcup R' = S\sqcup S'$ of the same finite set,
   the diagram
\be\label{hopf-diagram}
\begin{diagram}
\p[R] \otimes \p[R'] &\rTo^{\nabla_{R,R'}}& & \p[I] &&\rTo^{\Delta_{S,S'}} & \p[S]\otimes\p[S'] \\ 
\dTo^{\Delta_{A,B}\otimes \Delta_{A',B'}}&&  &&  && \uTo_{\nabla_{A,A'}\otimes \nabla_{B,B'}} \\ 
\p[A] \otimes \p[B] \otimes \p[A'] \otimes \p[B']&&& \rTo^{\sim} &&&\p[A] \otimes \p[A'] \otimes \p[B] \otimes \p[B'] 
 \end{diagram}
 \ee
 commutes, where 
$A = R \cap S 
$ and
$B = R\cap S'
$ and
$A' = R'\cap S
$ and
$B' = R' \cap S'.
$
Intuitively, if the product and coproduct are ``joining'' and ``splitting'' operations, then this condition means that we get the same thing by joining then splitting two objects as by splitting two objects into four  then joining appropriate pairs.
We now have a third definition:
\begin{itemize}

\item The  triple $(\p,\nabla,\Delta$) is a \emph{connected Hopf monoid} if the pair $(\p,\nabla)$ is a connected monoid and the pair $(\p,\Delta$) is a connected comonoid and $\nabla$ and $\Delta$ are Hopf compatible. A morphism of connected Hopf monoids is a morphism of connected species which commutes with  products and coproducts. 
A connected Hopf monoid is \emph{commutative} or \emph{cocommutative} if it is commutative as a monoid or cocommutative as a comonoid.
 \end{itemize}
The definition of a connected Hopf monoid given here
 is also that of a \emph{connected bimonoid} in \cite{species}. 
In a general symmetric monoidal category, a Hopf monoid is a bimonoid
for which an \emph{antipode}  exists (see \cite[Section 1.2.5]{species}).
A connected Hopf monoid in species always has an antipode; this is a certain morphism of connected species $\S : \p \to \p$, a general formula for which is given by \cite[Proposition 8.13]{species}.

\begin{remark}
There are more general notions of \emph{monoids}, \emph{comonoids}, and \emph{Hopf monoids} in species which may or may not be connected (see \cite[Chapter 8]{species})
and there are various natural functors from species to graded vector spaces which turn these respective structures into algebras,  coalgebras, and Hopf algebras (see Section \ref{fock-sect}).
Connected Hopf monoids are precisely the Hopf monoids which correspond to connected Hopf algebras via these functors.
\end{remark}

Each of the four connected species defined at the end of the previous section has a natural Hopf monoid structure, described as follows.
 
\begin{example}\label{E-ex}
$\bfE_C$ is a connected Hopf monoid 
with the product and coproduct given by
\[ \nabla_{S,T}(f\otimes g) = f\sqcup g\qquand \Delta_{S,T}(h) = h|_S \otimes h|_T\]
for maps $f : S \to C$ and $g : T \to C$ and $h : S\sqcup T \to C$.
Here $h|_S$ denotes the restriction of $h$ to $S$, while $f\sqcup g$ denotes the unique map $S\sqcup T \to C$ which restricts to $f$ on $S$ and to $g$ on $T$.
\end{example}

\begin{example}\label{Pi-ex}
$\bfPi$ is a connected Hopf monoid 
with the product and coproduct given by
\[ \nabla_{S,T}(X\otimes Y) = X \sqcup Y \qquand \Delta_{S,T}(Z) = Z|_S \otimes Z|_T\]
for partitions $X$, $Y$,  $Z$ of the sets $S$, $T$,  $S\sqcup T$. Here $Z|_S$ denotes the partition of $S$ induced by the partition $Z$ of the larger set $S\sqcup T$.
Explicitly, $Z|_S = \{ B \cap S : B \in Z \} - \{\varnothing\}$.
\end{example}


\begin{example}\label{L-ex}
$\bfL$ is a connected Hopf monoid
 with the product and coproduct given by
\[ \nabla_{S,T}(\ell_1\otimes \ell_2) = \ell_1 * \ell_2 \qquand \Delta_{S,T}(\ell) = \ell' \otimes \ell''\]
for linear orders $\ell_1$, $\ell_2$,  $\ell$ of the sets $S$, $T$,  $S\sqcup T$. Here $\ell'$ and $\ell''$ denote the linear order of $S$ and $T$ induced by $\ell$, and $\ell_1*\ell_2$ denotes the unique linear order of $S\sqcup T$ restricting to $\ell_1$ on $S$ and to $\ell_2$ on $T$, such that every element of $T$ exceeds every element of $S$. 

\end{example}

\begin{example}\label{S-ex}
$\bfSigma$ is a connected Hopf monoid with the product and coproduct given by
\[ \nabla_{S,T}(\sigma_1\otimes \sigma_2) = \sigma_1 \sqcup \sigma_2 \qquand \Delta_{S,T}(\sigma) = \sigma' \otimes \sigma''\]
for permutations $\sigma_1$, $\sigma_2$,  $\sigma$ of the sets $S$, $T$,  $S\sqcup T$.
Here $\sigma_1\sqcup \sigma_2$ is defined as in Example \ref{E-ex}, while $\sigma'$ is the permutation of $S$ such that if $i \in S$ then $\sigma'(i)$ is the first element of the sequence $\sigma(i),\ \sigma^2(i),\ \sigma^3(i),\ \dots$ belonging to $S$ (with $\sigma''$ defined likewise).
\end{example}

\section{Preliminaries on self-duality}
\label{prelim-sect2}

We review here the  results from \cite{Me} which 
 are relevant to the developments in the next section.   

\subsection{Linearization}
\label{lin-sect} 
We adapt the definitions in this section from \cite[Section 8.7]{species}.
A \emph{set species} is a functor $\P : \FB \to \Set$ where $\Set$ denotes the usual category of sets
with arbitrary maps as morphisms. 
Thus, $\P$ assigns, to each finite set $S$ and bijection $\sigma: S \to S'$, a set $\P[S]$ and a bijection $\P[\sigma] : \P[S] \to \P[S']$, subject   to  the conditions 
\[
\P[\id_S] = \id_{\P[S]}\qquand \P[\sigma\circ \sigma']= \P[\sigma]\circ \P[\sigma']
\]
whenever $\sigma$ and $\sigma'$ are composable bijections between finite sets.
We say that $\P$ is \emph{finite} if $\P[S]$ is  always finite. Likewise, we say that a vector species  $\p$ is \emph{finite-dimensional} if $\p[S]$ is always finite-dimensional.
If $\Q$ is another set species then we write $\Q \subset \P$ and say that $\Q$ is a \emph{subspecies}
of $\P$ if $\Q[S]\subset\P[S]$ for all finite sets $S$ and if $\Q[\sigma]$ is the restriction of $\P[\sigma]$  for all bijections $\sigma$. A vector species $\p$ is itself a set species, and  the notion of a subspecies of $\p$ is defined in the same way.

If $\P$ is a set species then we write $\kk\P$ for the vector species  such that 
$(\kk\P)[S]$ is the vector space with basis $\P[S]$
and
$(\kk\P)[\sigma]$ is the vector space isomorphism which is the linearization of $\P[\sigma]$.
If $\P[\varnothing] $ consists of exactly one element (which we will denote $1_\P$), then we view the  $\kk\P$
as a connected species  with respect to the unit $\kk \to \kk\P[\varnothing]$ and counit $\kk\P[\varnothing] \to \kk$ given by the  $\kk$-linear maps with 
$ 1_\kk \mapsto 1_\P$ and $  1_\P \mapsto 1_\kk$.

A connected species $\p$ is \emph{linearized} if there exists a set species $\P$ such that $\p = \kk\P$ as connected species. We refer to the set species $\P$ as a \emph{basis} for $\p$.
Explicitly, this means that each set $\P[S]$ is a basis for $\p[S]$, that each bijection $\P[\sigma]$ is the restriction of the linear map $\p[\sigma]$, and also that the unit and counit of $\p$ restrict to bijections $\{1_\kk\} \leftrightarrow \P[\varnothing]$.
If $\p$ is a vector species which is not necessarily connected, then 
we say that a set subspecies $\P\subset \p$ is a basis if merely each set $\P[S]$ is a basis for $\p[S]$.
A connected species has a basis if it has a basis when viewed as an arbitrary vector species, and a vector species $\p$ has a basis if and only if the representation of the symmetric group $S_n$ on $\p[n] = \p[\{1,\dots,n\}]$ is always a permutation representation  \cite[Proposition 3.1.1]{Me}.

\begin{notation} We denote by $\E$, $\E_C$, $\Pi$, $\L$, and $\fk S$ the set species which are the obvious bases of the connected species  $\bfE$, $\bfE_C$, $\bfPi$, $\bfL$, and $\bfSigma$.
For example, $\Pi$ is then the set species such that $\Pi[S]$ is the set of partitions of $S$.
\end{notation}

The main definition of this section is the following.

\begin{definition}\label{lin-def}
A connected Hopf monoid  $(\p,\nabla,\Delta)$ is \emph{linearized} in some basis $\P$ for $\p$ if for all disjoint finite sets $S$, $T$ the following conditions hold:
\begin{itemize}
\item $\nabla_{S,T}$ restricts to a map $\P[S]\times \P[T] \to \P[S\sqcup T]$.
\item $\Delta_{S,T}$ restricts to a map $\P[S\sqcup T] \to \P[S] \times \P[ T]$.

\end{itemize}
\end{definition}

All of the connected Hopf monoids given as examples in the previous section are linearized.
We refer to the first condition in Definition \ref{lin-def} as the requirement that the product $\nabla$ is linearized in the basis $\P$, and the second condition as the requirement that the coproduct $\Delta$ is linearized in $\P$. 
A connected monoid (respectively, comonoid) is \emph{linearized} in some basis if in this basis its product (respectively, coproduct) is linearized in this sense.

\subsection{Partial orders}

An important feature of a linearized Hopf monoid is its associated partial order, whose definition from \cite[Section 4.1]{Me} we review here.
If $\P$ is a set species, then let $\P\times \P$ denote the set species with $(\P\times \P)[S] = \P[S]\times \P[S]$ and $(\P\times \P)[\sigma] = \P[\sigma]\times \P[\sigma]$ for finite sets $S$ and bijections $\sigma$. 
We say that a subspecies $\cO \subset \P\times \P$ is \emph{transitive} if $(a,b),(b,c) \in \cO[S]$ implies $(a,c) \in \cO[S]$. A \emph{partial order} on a set species $\P$ is then a transitive subspecies $\cO\subset \P\times \P$ such that if $(a,b) \in \cO[I]$ then $a\neq b$.
%
When $\{ <\}$ is a partial order on $\P$ and  $\lambda,\lambda' \in \P[I]$ we write 
\[ \lambda < \lambda' \qquad\text{to indicate}\qquad (\lambda,\lambda') \in \{<\}[I].\]
In turn, we write $\lambda \leq \lambda'$ to mean $\lambda = \lambda'$ or $\lambda < \lambda'$.
With respect to this notation, $<$  corresponds to the usual notion of a partial order on each set $\P[I]$.
The condition that $\{<\}$ forms a subspecies of $\P\times \P$ is equivalent to requiring that $\P[\sigma](\lambda) < \P[\sigma](\lambda')$ whenever $\lambda < \lambda'$, for all bijections $\sigma : I \to I'$.

The following statement appears as \cite[Theorem G]{Me}. 

\begin{theorem}  \label{order-thm}
Let $\h= (\p,\nabla,\Delta)$ be connected Hopf monoid which is finite-dimensional, commutative, cocommutative, and linearized  in some basis $\P$.
 The (unique) minimal transitive subspecies $\{\prec\}\subset\P\times \P$ such that $(\lambda,\lambda') \in \{\prec\}[I]$   whenever $\lambda,\lambda' \in \P[I]$ are distinct elements and  \[\lambda = \nabla_{S,T}\circ \Delta_{S,T} (\lambda')\] for some disjoint decomposition $I = S\sqcup T$
is then a partial order on $\P$.
%
\end{theorem}

\begin{remark}
We may describe the order $\{\prec\}$ more explicitly as follows.
Let $\h = (\p,\nabla,\Delta)$ be an arbitrary connected Hopf monoid.
Suppose   $I = S_1 \sqcup S_2 \sqcup \cdots \sqcup S_k$ is a pairwise disjoint decomposition of a finite set.  There are $(k-1)!$ ways of composing  maps of the form $\id \otimes \cdots  \otimes \nabla_{X,Y}\otimes \cdots \otimes \id$ and maps of the form $\id \otimes \cdots  \otimes \Delta_{X,Y}\otimes \cdots \otimes \id$ to create $\kk$-linear maps
\[   \p[S_1] \otimes \p[S_2] \otimes \cdots \p[S_k] \to \p[I].
\qquand
   \p[I] \to \p[S_1] \otimes \p[S_2] \otimes \cdots \p[S_k]. \]
The associativity conditions required of $\nabla$ and $\Delta$
imply that  the compositions defining the left map all coincide and the compositions defining the right map all coincide; in this way we obtain two uniquely defined $\kk$-linear maps 
which we respectively denote 
\be\label{nablam}
\nabla_{S_1,S_2,\dots,S_k}
\qquand
\Delta_{S_1,S_2,\dots,S_k}.\ee When $k=1$   both of these maps are just the identity on $\p[I]$.
When $\h$ is linearized in some basis $\P$, the maps \eqref{nablam} descend to maps of sets (to and from $\P[I]$, respectively), and in the situation of Theorem \ref{order-thm} the corresponding order $\{\prec\}$
 has this description:
if $\lambda,\lambda' \in \P[I]$ then 
\[\lambda \preceq \lambda'\qquad\text{if and only if}\qquad
\lambda = \nabla_{S_1,\dots,S_k}\circ \Delta_{S_1,\dots,S_k}(\lambda')\]
for some pairwise disjoint decomposition $I = S_1\sqcup \dots \sqcup S_k$.
\end{remark}

The theorem applies when
 $\h = \bfPi$ and $\P = \Pi$ is the natural basis of set partitions. In this case the corresponding partial order  on $\Pi$  is the ordering of set partitions by \emph{refinement}. Recall that if $X$, $Y$ are two partitions of the same set then we say that $X$ \emph{refines} $Y$ and we write  $X \leq Y$ if each block of $X$ is a subset of some block of $Y$.
 
 In the following section we will find that in the situation of Theorem \ref{order-thm} we may always identify the basis $\P$ with some species of ``labeled set partitions,'' and via this identification the order $\{\prec\}$ is a generalization of the refinement order.

%
%

%
%

\subsection{Strong self-duality}\label{ssd-sect}

The \emph{dual} of a vector species $\p$ is the vector species $\p^*$ with $\p^*[S] =( \p[S])^*$ and $\p[\sigma] =( \p[\sigma^{-1}])^*$ for finite sets $S$ and bijections $\sigma$,
where 
we write $V^*$ for the usual dual space of a vector space $V$ and $f^* : W^* \to V^*$ for the usual transpose of a linear map $f : V \to W$.
If $\h = (\p,\nabla,\Delta)$ is a connected Hopf monoid,
then its \emph{dual} is the connected Hopf monoid $\h^*$ whose underlying species is $\p^*$ and whose product and coproduct  
are the systems of maps comprised of the respective compositions
\[   \p^*[S]\otimes \p^*[ T]\xrightarrow{\sim} \(\p[S]\otimes \p[T]\)^*  \xrightarrow{(\Delta_{S,T})^*} \p^*[S\sqcup T]\]
and
\[   \p^*[S\sqcup T] \xrightarrow{(\nabla_{S,T})^*} \(\p[S]\otimes \p[T]\)^* \xrightarrow{\sim} \p^*[S] \otimes \p^*[T]\]
for disjoint finite sets $S$, $T$. Here, the maps $\xrightarrow{\sim}$ denote the  canonical isomorphisms $(V\otimes W)^* \cong V^* \otimes W^*$ for vector spaces $V$, $W$.
For a more extensive discussion of duality in species, see \cite[Section 2.3]{Me} or \cite[Section 8.6]{species}.

A connected Hopf monoid $\h = (\p,\nabla,\Delta)$ is \emph{self-dual} if it is isomorphic to its dual $\h^*$. The  following variation of Definition \ref{lin-def} turns out to give a property strictly stronger than self-duality, whose discussion will be our principal concern in this section. 

\begin{definition}\label{ssd-def} A connected Hopf monoid $(\p,\nabla,\Delta)$ is \emph{strongly self-dual} (or \emph{SSD} for short) in some basis $\P$ for $\p$ if 
for all disjoint finite sets $S$, $T$ the following conditions hold:
\begin{itemize}
\item $\nabla_{S,T}$ restricts to a map $\P[S]\times \P[T] \to \P[S\sqcup T]$.
\item $\P[S\sqcup T]$ is a  finite subset of the union of the image of $\nabla_{S,T}$ and the kernel of $\Delta_{S,T}$.
\end{itemize}
\end{definition}

Such Hopf monoids are considered in detail in \cite[Section 3.3]{Me}; they include Examples \ref{E-ex}, \ref{Pi-ex}, and \ref{S-ex} (though these Hopf monoids are not strongly self-dual in their obvious bases). It follows by \cite[Proposition 3.3.5]{Me} that
a connected Hopf monoid which is strongly self-dual is self-dual in the ordinary sense given above, and also has the  property of being \emph{freely self-dual} (as defined in \cite[Section 3]{Me}) in the sense that the structure constants of its product and coproduct  coincide for its distinguished basis.
Note that if $(\p,\nabla,\Delta)$ is   strongly self-dual then $\p$ must be finite-dimensional.

It turns out that all strongly self-dual  Hopf monoids  originate from the following construction.
Let $\Q$ be a set species.
A \emph{$\Q$-labeled set partition} $X$ is  
a
  set of pairs $(B,\lambda)$, which we call the \emph{blocks} of $X$, consisting of a nonempty finite set  $B$ and an element  $\lambda \in \Q[B]$, subject to the condition that 
$B \cap B' = \varnothing$
whenever 
$(B,\lambda)$ and $(B',\lambda')$ are distinct
blocks. 
The \emph{shape} of a labeled set partition $X$ is the unlabeled set partition
\be\label{shape} \sh(X) \omdef= \{B : (B,\lambda) \in X\}.\ee
We   say that $X$ is a partition of $S$ if $\sh(X) \in \Pi[S]$.
%
The collection of $\Q$-labeled set partitions forms a set species, which we denote $\cS(\Q)$, in the following way: for each finite set $I$ define 
$\cS(\Q)[I]$
as the set of $\Q$-labeled partitions of $I$,
and
for each bijection $\sigma : I \to I'$ between finite sets 
define $\cS(\Q)[\sigma]$
as the map sending $X \in \cS(\Q)[I]$ to the labeled partition $X' \in \cS(\Q)[I']$ such that 
\[(B,\lambda)\in X
\qquad\text{if and only if}
\qquad 
 (\sigma(B), \Q[\sigma](\lambda))\in X'.\]
Note that $\cS(\Q)$ has no dependence on the set $\Q[\varnothing]$.

Since $\cS(\Q)[\varnothing] = \{\varnothing\}$, the vector species $\kk\cS(\Q)$ 
is connected.
We view this connected species as a Hopf monoid with respect to the following product and coproduct.
Define $\nabla$ as the product on $\kk\cS(\Q)$ such that 
$\nabla_{S,T}(X\otimes Y )  = X\sqcup Y$
whenever $X$ and $Y$  are $\Q$-labeled partitions of disjoint finite sets $S$ and $T$.
Define $\Delta$ as the coproduct on $\kk\cS(\Q)$ 
such that if $S$, $T$ are disjoint finite sets and $Z \in \cS(\Q)[S\sqcup T]$ then
\[
\Delta_{S,T}(Z) = \begin{cases} X\otimes Y&\text{if there are $X \in \cS(\Q)[S]$ and $Y\in \cS(\Q)[T]$ with $Z = X \sqcup Y$} \\ 0&\text{otherwise}.\end{cases}
\]
The next fact is a rewording of \cite[Corollary 3.3.7]{Me} and also follows from more general results discussed in \cite[Section 11.3]{species}.
\begin{fact}\label{ssd-fact} 
With $\nabla$ and $\Delta$ defined as above, the triple $(\kk\cS(\Q),\nabla,\Delta)$ is a connected Hopf monoid. This connected Hopf monoid  is strongly self-dual whenever $\Q$ is finite.
\end{fact}

%
%
%

Given a connected monoid $(\p,\nabla)$ and a set subspecies $\P \subset \p$, define $\cP(\P,\nabla)$ as the set subspecies of $\P$ with $\cP(\P,\nabla)[\varnothing] = \varnothing$ and 
with $\cP(\P,\nabla)[I]$ comprised, for nonempty sets $I$, of the elements in $\P[I]$ which are not in the image of $\nabla_{S,T}$ for any disjoint decomposition $I =S\sqcup T$ into two nonempty subsets.
The next statement paraphrases \cite[Theorem 3.3.11]{Me}.

\begin{theorem}\label{ssd-thm}
Let $\h = (\p,\nabla,\Delta)$ be a connected Hopf monoid which is strongly self-dual in some basis $\P$,
and write $\Q = \cP(\P,\nabla)$.
Then
$\h$ is isomorphic to $\kk\cS(\Q)$ with the Hopf monoid structure described in Fact \ref{ssd-fact}.
In particular, $\h$ is both commutative and cocommutative. 
\end{theorem}

Returning our attention   to linearized Hopf monoids, we observe the following statement which appears as \cite[Theorem F]{Me}.

\begin{theorem}\label{whenfsd-thm}
A connected Hopf monoid which is linearized in some basis is strongly self-dual in some (possibly different) basis if and only if it is finite-dimensional, commutative, and cocommutative. 
\end{theorem}

 
 Conversely,
  one can classify
  which strongly self-dual  Hopf monoids are linearized in terms of the existence of a certain partial order.  A precise statement of such a classification appears as \cite[Theorem 4.1.9]{Me}.
   In detail,  by \cite[Lemma 4.1.7]{Me} the partial order $\{ \prec\}$  given in Theorem \ref{order-thm} has the following  properties:
  (A)   lower intervals are preserved by products
and
    (B) if  $I = S\sqcup T$ is a disjoint decomposition then each $\lambda \in \P[I]$ has a unique greatest lower bound in the image of $\nabla_{S,T}$.
A connected Hopf monoid $\h = (\p,\nabla,\Delta)$ which is strongly self-dual in some basis $\P$ then has 
a basis in which it is linearized if and only if there is a partial order $\{\prec\}$ on $\P$  with both of these properties.

\section{Structure of strongly linearized Hopf monoids}
\label{structu-sect}

In this section we begin to state new results, related to the construction and classification
of a certain type of linearized Hopf monoid (which we call \emph{strongly linearized}) that will give rise to many Hopf algebras of interest in Section \ref{app-sect}. 


%

\subsection{Linearized Hopf monoids from comonoids}
  \label{co2hopf}


From  a connected comonoid $(\q,\Delta)$ which is linearized in some basis $\Q$, we may   construct a  connected Hopf monoid $\cS(\q,\Delta)$ which is linearized  in the following way.
The underlying species structure of $\cS(\q,\Delta)$ will be the  connected species $\kk\cS(\Q)$ where $\cS(\Q)$ is the   species of $\Q$-labeled set partitions described after Definition \ref{ssd-def},
and the product $\nabla$ will be the same as in Fact \ref{ssd-fact}.
It remains to define the coproduct.

With some abuse of notation, we will denote the coproduct of $\cS(\q,\Delta)$ also by $\Delta$.
In words, each component of the coproduct of $\cS(\q,\Delta)$  divides a labeled set partition into two by splitting each of its blocks $(B,\lambda)$; blocks are split
by taking the intersection of $B$ with  $S_i$ and then labeling these intersections by the components of the corresponding image of $\lambda$ under $\Delta$. 
Explicitly, if $I = S_1\sqcup S_2$ is a disjoint decomposition of a finite set and 
$Z = \{ (B_j,\lambda_j)\}_{j\in J} \in \cS(\Q)[I]$
is a $\Q$-labeled set partition,
then for each index $j \in J$ there are unique elements 
\[\lambda_{1j} \in \Q[S_1\cap B_j]
\qquand
\lambda_{2j} \in \Q[S_2 \cap B_j]\]
such that
$
    \Delta_{S_1 \cap B_j, S_2 \cap B_j} (\lambda_j) = \lambda_{1j} \otimes \lambda_{2j}
$
. From $Z$ we thus obtain two $\Q$-labeled set partitions
\[ Z_i = \{ ( S_i \cap B_j, \lambda_{ij}) : \text{$j \in J$ such that $S_i \cap B_j \neq \varnothing$}\} \in \cS(\Q)[S_i]\qquad\text{for $i \in \{1,2\}$}
\]
and to define $\Delta_{S_1,S_2}$ on $\kk\cS(\Q)[I]$ we set  $ \Delta_{S_1,S_2}(Z) = Z_1\otimes Z_2$  and then extend by linearity.

\begin{remark}
There is a natural inclusion $\Q \subset \cS(\Q)$ given by identifying $\Q$ with the subspecies of $\Q$-labeled partitions with at most one block. This gives rise to an inclusion $\q \subset \kk\cS(\Q)$, and the action of the coproduct just defined on this subspecies may be identified with the original action of $\Delta$ on $\q$. In this way our new coproduct  is an extension of $\Delta$, which justifies in part our use of the same symbol $\Delta$ to denote it.
 Moreover, there is a sense in which this extension is the unique one which makes 
$\kk\cS(\Q)$ into a connected Hopf monoid with the product $\nabla$.
\end{remark}

Set $\cS(\q,\Delta) = (\kk\cS(\Q), \nabla,\Delta)$. Let $\q_+$ denote the maximal subspecies of $\q$ with $\q_+[\varnothing] =0$. Then $(\q_+,\Delta)$ is a \emph{positive comonoid} in the sense of \cite[Definition 8.42]{species}, and $\cS(\q,\Delta)$ is the \emph{free commutative Hopf monoid} on this positive comonoid, as described in \cite[Section 11.3.2]{species}. The following fact derives from the discussion in that section.

\begin{fact}\label{sl-fact} 
With $\nabla$ and $\Delta$ defined as above, $\cS(\q,\Delta)$ 
is a connected Hopf monoid which is 
commutative and linearized in the basis $\cS(\Q)$. Up to isomorphism,  $\cS(\q,\Delta)$ does not depend on the choice of basis $\Q$. Moreover, $\cS(\q,\Delta)$ is cocommutative if  and only if $(\q,\Delta)$ is cocommutative.
\end{fact}

\begin{remark}
We have given the somewhat simplified definition of $\cS(\q,\Delta)$ in the case that $(\q,\Delta)$ is linearized. One can extend this definition to construct a commutative connected Hopf monoid from any connected comonoid; see
\cite[Section 11.3.2]{species}.
\end{remark}
%
%
 
Assume $(\q,\Delta)$ is cocommutative and linearized in some finite basis $\Q$, and let $\{\prec\}$ be the partial order on $\cS(\Q)$ defined from the product and coproduct of $\cS(\q,\Delta)$ by Theorem \ref{order-thm}. 
The order $\{\prec\}$ has the following explicit definition: if $X$ and $Y$ are two $\Q$-labeled partitions of the same finite set then 
$ X \preceq Y $
if and only if whenever $(B,\lambda) \in Y$ there are blocks $(A_i,\lambda_i) \in X$ such that 
\[ B = A_1\sqcup\dots \sqcup A_k\qquand \lambda = \Delta_{A_1,\dots,A_k}(\lambda_1\otimes \cdots\otimes\lambda_k).\]
In particular, this holds  only if  $\sh(X) \leq \sh(Y)$ in the ordering of set partitions by refinement, where  $\sh(X) = \{ B : (B,\lambda) \in X\}$. Moreover, it follows as a straightforward exercise   that
 the map  $ X \mapsto \sh(X)$
 gives an isomorphism of partially ordered sets
\[\{ X \in \cS(\Q)[I] : X \preceq Y\} \xrightarrow{\sim}  \{ \Lambda \in \Pi[I] : \Lambda \leq \sh(Y)\}.\]
Theorem \ref{posetisom-thm} below will generalize this observation.

\begin{remark}
It is interesting to compare the partial order $\{\prec\}$ with the one described in \cite[Section 8.7.6]{species}, from the work of M\'endez and Yang \cite{Mendez0,Mendez}.
The latter  order is defined on the natural basis of the \emph{Cauchy product} of the exponential species $\bfE$ with a linearized connected monoid; the partial order here, by contrast, is defined on the natural basis of the \emph{substitution product} of $\bfE$ with a linearized positive comonoid. (See \cite[Definition 8.5]{species} for an explanation of these products.)
\end{remark}

\subsection{Strongly linearized Hopf monoids}

The Hopf monoid $\cS(\q,\Delta)$ exhibits the following strong form of linearization which,
conversely, will in some sense characterize all  Hopf monoids arising from this construction.
\begin{definition}\label{sl-def} A connected Hopf monoid $(\p,\nabla,\Delta)$ is 
 \emph{strongly linearized} in some basis $\P$ for $\p$ if it is linearized in the basis $\P$ and if  $\Delta_{S,T}$ restricts to a map $\Q[S\sqcup T] \to \Q[S]\times \Q[T] $ for all disjoint nonempty finite sets $S$, $T$,
 where $\Q = \cP(\P,\nabla)$.
\end{definition}

The Hopf monoids in Examples \ref{E-ex}, \ref{Pi-ex}, \ref{L-ex}, and  \ref{S-ex} all have this property in their natural bases.
For the rest of this section,  let $\h = (\p,\nabla,\Delta)$ denote  a fixed connected Hopf monoid which is linearized in some finite basis $\P$.
Also set $\Q = \cP(\P,\nabla)$ and $\q = \kk\Q$. Define $\q^\circ$ to be the  subspecies of $\p$ with 
\[ \q^\circ[\varnothing] = \p[\varnothing]\qquand \q^\circ [S] = \q[S]\text{ for nonempty sets $S$.}\]
We consider $\q^\circ$ to be connected with the same unit and counit as $\p$.
Finally, when $\h$ is commutative and cocommutative, we write $\{\prec\}$ for the partial order   on $\P$ defined in Theorem \ref{order-thm}.
In this setup, we have the following variation of Theorem \ref{ssd-thm}. 
\begin{theorem}\label{fl-thm}
Suppose $\h$ is commutative and strongly linearized with respect to $\P$. Restricting  $\Delta$ then gives $\q^\circ$ the structure of a linearized connected comonoid, and there is a unique isomorphism of connected Hopf monoids $\cS(\q^\circ,\Delta) \xrightarrow{\sim} \h$ whose restriction makes the diagram
\[
\begin{diagram}
\cS(\Q) && \rTo^{\sim} && \P \\
& \luTo &&\ruTo\\
&& \Q
 \end{diagram}
\]
 commute, where the diagonal arrows are the natural inclusions.
 \end{theorem}

\begin{proof}
By \cite[Lemma 4.1.3]{Me} and \cite[Theorem 3.3.11]{Me}, there exists a unique isomorphism of connected \emph{monoids} $f : \cS(\q^\circ,\Delta) \to \h$
whose restriction makes the given diagram commute.
It remains to check that $f$ is also a morphism of connected comonoids, i.e., that 
\be\label{com} (f_{S}\otimes f_{S'})\circ \Delta_{S,S'} = \Delta_{S,S'}\circ f_I\ee
for any disjoint decomposition $I = S\sqcup S'$, where we write $\Delta$ for the coproducts of both $\cS(\q^\circ,\Delta)$ and $\h$.
This identity holds automatically if $S$ or $S'$ is empty since $f$ is an isomorphism of connected species. We proceed by induction on $|S| + |S'|$, so assume $S$ and $S'$ are both nonempty and
  that \eqref{com} holds if either set is replaced by a set with fewer elements.

Since $\h$ is strongly linearized and since $f$ commutes with the inclusions $\Q \to \cS(\Q)$ and $\Q \to\P$, it follows from the definition of the coproduct on $\cS(\q^\circ,\Delta)$ that both sides of \eqref{com}  agree when applied to a $\Q$-labeled partition with at most one block. 
To prove \eqref{com} it remains only to show that both sides of \eqref{com} agree when applied to $Z \in \cS(\Q)[I]$ with more than one block. 
Writing $\nabla$ for the product of $\cS(\q^\circ,\Delta$), we note that such a partition belongs to the image of $\nabla_{R,R'} $ for some disjoint decomposition $I = R\sqcup R'$ with $R$ and $R'$ both nonempty. The identity \eqref{com} will therefore follow if we show that for such a decomposition
\be\label{com2} (f_{S}\otimes f_{S'})\circ \Delta_{S,S'}\circ \nabla_{R,R'} = \Delta_{S,S'}\circ f_I\circ \nabla_{R,R'}.\ee
For this, define 
$A = R \cap S 
$ and
$B = R\cap S'
$ and
$A' = R'\cap S
$ and
$B' = R' \cap S'
$.
The Hopf compatibility of the product and coproduct of $\cS(\q^\circ,\Delta)$ then implies that the left side of \eqref{com2} is equal to
\[ 
 (f_{S}\otimes f_{S'})\circ (\nabla_{A,A'} \otimes \nabla_{B,B'}) \circ \tau \circ (\Delta_{A,B}\otimes \Delta_{A',B'}).
 \]
 where $\tau$ is an appropriate twisting bijection (see the diagram \eqref{hopf-diagram}).
 Since $f$ is a morphism of connected monoids, this expression  is equal to
\[
 (\nabla_{A,A'} \otimes \nabla_{B,B'}) \circ   (f_{A}\otimes f_{A'}\otimes f_{B}\otimes f_{B'})\circ \tau \circ (\Delta_{A,B}\otimes \Delta_{A',B'})
 \]
which is equal in turn, after commuting the two middle maps, to 
\[
 (\nabla_{A,A'} \otimes \nabla_{B,B'}) \circ \tau'
 \circ  ((f_{A}\otimes f_B)\circ \Delta_{A,B})\otimes ((f_{A'}\otimes f_{B'})\circ \Delta_{A',B'})
 \]
 where $\tau'$ is now another appropriately chosen twisting bijection.
 By hypothesis, since $R$ and $R'$ are both nonempty, we have $(f_{A}\otimes f_B)\circ \Delta_{A,B} = \Delta_{A,B}\circ f_R$ and 
 $(f_{A'}\otimes f_{B'})\circ \Delta_{A',B'} = \Delta_{A',B'} \circ f_{R'}$. Substituting these identities and again invoking Hopf compatibility shows that the previous expression is equal to
\[  (\nabla_{A,A'} \otimes \nabla_{B,B'}) \circ \tau'
 \circ  ( \Delta_{A,B}\otimes \Delta_{A',B'}) \circ (f_R\otimes f_{R'}) =    \Delta_{S,S'}\circ \nabla_{R,R'}\circ   (f_R\otimes f_{R'}) .\]
  The right expression here coincides with the right expression in \eqref{com2} since $f$ commutes with products. Combining these observations, we conclude by induction that \eqref{com2} holds, and hence more generally that \eqref{com} holds.
\end{proof}

The next theorem shows that when  $\{\prec\} \subset \P\times \P$ is defined, we may distinguish when the Hopf monoid $\h$ is strongly linearized, as opposed to being merely linearized, by a property of this associated partial order.

\begin{theorem}\label{posetisom-thm} 
Assume $\h$ is commutative and cocommutative. Then $\h$ is strongly linearized with respect to $\P$ if and only if 
whenever  $I$ is  a finite set and $\lambda \in \Q[I]$, the sets 
\[\label{thesets}\{ \lambda' \in \P[I] : \lambda' \preceq \lambda\}
\qquand \Pi[I]\] have the same number of elements. 
In this case, furthermore,
there exists an order-preserving bijection between these   sets, with $\Pi[I]$  ordered by refinement.

\end{theorem}

The following notation will be of use in the proof of this result. Recall from \eqref{shape} the definition of the shape of a labeled set partition.
When $\h$ is  commutative and strongly linearized in $\P$ and 
$I$ is a finite set, we  define the \emph{shape} of $\lambda \in \P[I]$ to be the set partition 
\[ \sh(\lambda) = \sh(X) \in \Pi[I]\] where $X \in \cS(\Q)[I]$ is the preimage of $\lambda$ under
the
isomorphism $\cS(\q^\circ,\Delta) \xrightarrow{\sim} \h$   in Theorem \ref{fl-thm}.

\begin{proof}
Fix a finite set $I$ and an element $\lambda \in \Q[I]$, and let $L= \{ \lambda ' \in \P[I] : \lambda' \preceq \lambda\}$. 
Given a set partition $X = \{ B_1,\dots,B_k\} \in \Pi[I]$ with $k$ blocks,
let 
\[ \lambda_X = \nabla_{B_1,\dots,B_k} \circ \Delta_{B_1,\dots,B_k}(\lambda).\]
This element is well-defined, independent of the ordering of the blocks of $X$, since $\h$ is commutative and cocommutative. It follows from the remarks after Theorem \ref{order-thm} that 
$ L = \{ \lambda_X : X \in \Pi[I]\}$, and so
we have the inequality $|L| \leq |\Pi[I]|$.

Suppose $\h$ is not strongly linearized in $\P$, so that for some choice of $\lambda$ there exists a disjoint decomposition $I=S\sqcup S'$ with $S$ and $S'$ nonempty such that $\Delta_{S,S'}(\lambda) = (\lambda',\lambda'') \notin \Q[S]\times \Q[S']$.
Without loss of generality assume $\lambda' \in \P[S]-\Q[S]$.
By \cite[Lemma 4.1.3]{Me} and \cite[Theorem 3.3.11]{Me}
as in the last theorem, it follows that $\lambda'  = \nabla_{R,R'}(\alpha\otimes \alpha')$ for some disjoint decomposition $S = R\sqcup R'$ with $R$ and $R'$ nonempty and some elements $(\alpha,\alpha') \in \P[R]\times \P[R']$.
As the composition $\Delta_{R,R'}\circ \nabla_{R,R'}$ is always the identity map \cite[Lemma 2.2.4]{Me},
it follows that 
\[ \Delta_{R,R',S'}(\lambda) = (\Delta_{R,R'}\otimes \id)\circ \Delta_{S,S'}(\lambda) = \alpha\otimes \alpha'\otimes \lambda''\]
and hence
\[ \nabla_{R,R',S'}\circ \Delta_{R,R',S'}(\lambda) = \nabla_{S,S'} \circ (\nabla_{R,R'}\otimes \id)( \alpha\otimes \alpha'\otimes \lambda'') = \nabla_{S,S'}(\lambda'\otimes \lambda'').\]
Consequently, we have $\lambda_X = \lambda_Y$ for the distinct set partitions $X = \{ S,S'\}$ and $Y = \{R,R',S'\}$, and so the inequality $|L| \leq |\Pi[I]|$ must be strict.

Conversely, if $\h$ is strongly linearized in $\P$, then for any pairwise disjoint decomposition $I=B_1\sqcup \dots \sqcup B_k$ into nonempty subsets,  $\Delta_{B_1,\dots,B_k}$ restricts to a map 
$\Q[I] \to \Q[B_1]\times \dots \times \Q[B_k]$
and it follows that
  $\sh(\lambda_X) = X$ for all $X \in \Pi[I]$. In this case, therefore, the map   $\lambda' \mapsto \sh(\lambda')$ defines   a bijection $L\to \Pi[I] $.
By \cite[Proposition 4.1.8]{Me} this map is also order-preserving, which proves the last sentence in the theorem.
\end{proof}

%
%
%

 \subsection{Basis constructions}\label{bases-sect}

Everywhere in this section $\h = (\p,\nabla,\Delta)$ denotes  a fixed connected Hopf monoid which is 
 finite-dimensional, commutative, cocommutative, and strongly linearized in some basis $\P$ for $\p$.
We write $\{\prec\}$ for the partial order   on $\P$ defined via Theorem \ref{order-thm}.
The goal of this section is to introduce four distinguished bases for this Hopf monoid.
Examples in the next section will illustrate a sense in which these bases 
generalize the definition of the classical bases of the Hopf algebra of symmetric functions.

For any finite set $I$, we write $\Mobius(\cdot,\cdot)$ for the  M\"obius function on $\P[I]$ ordered by $ \preceq$. Recall that this is the unique function $\P[I] \times \P[I] \to \ZZ$
such that 
\[ \Mobius(\lambda',\lambda'') = 0\text{ if }\lambda'\not\preceq \lambda''\qquand 
\sum_{\substack{\lambda'' \in \P[I] \\ \lambda ' \preceq \lambda'' \preceq \lambda} } \Mobius(\lambda',\lambda'') = \begin{cases} 1 &\text{if }\lambda=\lambda' \\ 0&\text{otherwise}.\end{cases}\]
The identity on the right implies $\Mobius(\lambda,\lambda) = 1$
and provides a recurrence for computing any value of the M\"obius function.
Fix 
an element $\lambda \in \P[I]$. We first define
\[
p_\lambda = \sum_{\lambda' \preceq \lambda} \Mobius(\lambda',\lambda)\cdot \lambda'
\qquand
m_\lambda = \sum_{ \lambda' \succeq \lambda} \Mobius(\lambda,\lambda') \cdot p_{\lambda'}
\]
where both sums are over  $\lambda' \in \P[I]$.
By Theorem \ref{posetisom-thm} 
 the poset  $\{ \lambda' \in \P[I] : \lambda' \preceq \lambda\}$ has a unique minimal element. We denote this element by $\emptyset_\lambda$
and set $\Mobius(\lambda)  = \Mobius(\emptyset_\lambda,\lambda)$.
Next   define
\[
e_\lambda = \sum_{\lambda' \preceq \lambda} \Mobius(\lambda')\cdot p_{\lambda'}
\qquand
h_\lambda = \sum_{ \lambda' \preceq \lambda} |\Mobius(\lambda')| \cdot p_{\lambda'}
\]
To refer to the collections of these elements,  for finite sets $I$  we 
set $x_\bullet[I] = \{ x_\lambda : \lambda \in \P[I]\}$ when $x$ is one of the letters $m$, $p$, $e$, $h$. For each of these letters, $x_\bullet$ then forms a set subspecies of $\p$, by our definition of a partial order.

\begin{lemma}\label{basis-lem1}
If $I$ is a finite set and $\lambda \in \P[I]$ then 
$\lambda = \sum_{\lambda' \preceq\lambda} p_{\lambda'}$
and
$p_\lambda = \sum_{\lambda' \succeq \lambda} m_{\lambda'}$.
\end{lemma}

\begin{proof}
Substitute into the right hand expressions our original formulas for $p_{\lambda'}$ and $m_{\lambda'}$. Then apply the identities defining the M\"obius function.
\end{proof}

Recall, from the remarks preceding the proof of Theorem \ref{posetisom-thm}, the definition of the set partition $\sh(\lambda) \in \Pi[I]$ which gives the shape of a basis element $\lambda \in \P[I]$. 

\begin{lemma}\label{basis-lem2}
If $I$ is a finite set and $\lambda \in \P[I]$ then $\Mobius(\lambda) =\ds \prod_{B \in \sh(\lambda)} (-1)^{|B|-1} \cdot  (|B|-1)!$.
\end{lemma}


\begin{proof}
We claim that the interval $[\emptyset_\lambda,\lambda]$ in $\P[I]$ ordered by $\prec$ is isomorphic to the interval $[\hat 0, \sh(\lambda)]$
in the lattice of set partitions $\Pi[I]$ ordered by refinement. (Here we write $\hat 0$ for the unique minimum in $\Pi[I]$, namely, the partition into blocks all of size one.)
This holds if $\lambda \in \cP(\P,\nabla)[I]$ by Theorem \ref{posetisom-thm}.
If $\lambda \notin \cP(\P,\nabla)[I]$ then 
\[\lambda = \alpha\cdot \beta
\qquand \sh(\lambda) =\sh(\alpha)\sqcup \sh(\beta)\] by Theorem \ref{fl-thm} for some $(\alpha,\beta) \in \P[S]\times \P[T]$ where $I=S\sqcup T$ is a disjoint decomposition into nonempty subsets.
Since products preserve lower intervals in the $\P[I]$ (see \cite[Theorem 4.1.9]{Me}), we have 
$ [\emptyset_\lambda,\lambda] \cong [\emptyset_\alpha,\alpha]\times [\emptyset_\beta,\beta].$
On the other hand, in the lattice of set partitions we have
$
[\hat 0,\sh(\lambda)] \cong  [\hat0,\sh(\alpha)]\times [\hat0,\sh(\beta)]
$ 
since more generally $[\hat 0,\sh(\lambda)] \cong \prod_{B \in \sh(\lambda)} \Pi[B] $, and so the desired isomorphism  follows by induction.
By our claim, we lose no generality in assuming 
 $\P = \Pi$, and in this special case it is well-known that the value of $\Mobius(\lambda)$ is equal to the given formula; see \cite[Section 3]{Sagan1}.
\end{proof}

We introduce the following shorthand for the product and coproduct of $\h$ on basis elements. Given disjoint finite sets $S$, $T$ and elements $\alpha \in \P[S]$, $\beta\in \P[T]$, and $\lambda \in \P[S\sqcup T]$ we write 
\[ \alpha\cdot \beta \in \P[S\sqcup T] \qquand \lambda|_S \in \P[S]\qquand \lambda|_T \in \P[T]\]
for the elements such that $\nabla_{S,T}(\alpha,\beta) = \alpha\cdot \beta$ and $\Delta_{S,T}(\lambda) = \lambda|_S\otimes \lambda|_T$. Observe that the notation $\lambda|_S$ is well-defined since the coproduct is cocommutative.

\begin{theorem}\label{basis-thm}
Each of $m_\bullet$, $p_\bullet$, $e_\bullet$, and $h_\bullet$ is a basis for $\p$.
On these bases, for disjoint finite sets $S$ and $T$ and indices $\alpha \in \P[S]$ and $\beta \in \P[T]$ and $\lambda \in \P[S\sqcup T]$, the product and coproduct of $\h$   have the following formulas:
\ben
 
\item[(a)] When $x \in \{p,e,h\}$ and $y \in \{m\}$, it holds 
that
\[\nabla_{S,T}(x_\alpha\otimes x_\beta) = x_{\alpha\cdot \beta} \qquand \nabla_{S,T}(y_\alpha\otimes y_\beta) =\ds \sum_{\substack{\gamma \in \P[S\sqcup T] \\ \gamma|_S=\alpha\text{ and }\gamma|_T= \beta}} y_\gamma.\]

\item[(b)] When $x \in \{e,h\}$ and $y \in \{m,p\}$, it holds 
that
\[
\Delta_{S,T}(x_\lambda) = x_{\lambda|_S} \otimes x_{\lambda|_T}
\qquand
\Delta_{S,T}(y_\lambda)  
=
\begin{cases}   y_{\lambda|_S} \otimes y_{\lambda|_T} &\text{if $\lambda \in \im(\nabla_{S,T})$} \\ 0&\text{otherwise}.\end{cases}\]
\een
It follows, in particular, that
 the natural changes of basis $p_\bullet \leftrightarrow h_\bullet$, $p_\bullet \leftrightarrow m_\bullet$, and $e_\bullet \leftrightarrow h_\bullet$ are respectively isomorphisms of connected monoids, comonoids, and Hopf monoids. Moreover, $\h$ is strongly linearized with respect to the bases $e_\bullet$ and $h_\bullet$, and strongly self-dual with respect to $p_\bullet$.
\end{theorem}

\begin{proof}
It is clear from Lemma \ref{basis-lem1} that the subspecies $m_\bullet$ and $p_\bullet$ are bases of $\kk\P$. In turn, it follows that $e_\bullet$ and $h_\bullet$ are bases since by Lemma \ref{basis-lem2} it never occurs that $\Mobius(\lambda)=0$.
The formulas for the product and coproduct on the $m_\bullet$ and $p_\bullet$ bases follow from \cite[Theorem 4.2.1]{Me}. (Note, however, that the notation of that result is dual to what we employ here. In \cite[Section 4.2]{Me}, what we denote by $\lambda$ is called $m_\lambda$ and what we denote by  $\lambda$ is called $q_\lambda$.)
To check the remaining formulas, it suffices to show that 
the systems of $\kk$-linear maps $\p[I] \to \p[I]$ with 
\be\label{formulas}p_\lambda \mapsto \Mobius(\lambda)\cdot p_\lambda \qquand p_\lambda \mapsto |\Mobius(\lambda)|\cdot p_\lambda 
\qquad\text{for }\lambda \in \P[I]\ee
both give endomorphisms  of the connected Hopf monoid $\h$. For, these endomorphisms are clearly automorphisms since $\Mobius(\lambda)$ is never zero,
and by Lemma \ref{basis-lem2} they define the natural changes of basis $\P \to e_\bullet$ and $\P \to h_\bullet$.
Hence our claim would imply that the product and coproduct of $\h$ act on $e_\bullet$ and $h_\bullet$ by the same formulas as on the original basis $\P$, which is what the theorem asserts.

It follows from our definition of partial orders for set species that both formulas in \eqref{formulas} extend at least to morphisms $\p \to \p$ of connected species.  Moreover, by Lemma \ref{basis-lem2} we have $\Mobius(\alpha\cdot \beta) = \Mobius(\alpha) \cdot \Mobius(\beta)$ whenever $(\alpha,\beta) \in \P[S]\times\P[T]$ for disjoint finite sets $S$, $T$.
Combining this fact with the formulas for the product and coproduct of $\h$ in the  $p_\bullet$ basis
 shows that the morphisms of connected species \eqref{formulas} commute with the product and coproduct of $\h$,
 and so are morphisms of Hopf monoids.
  \end{proof}

Define the \emph{sign} of a set partition $X \in \Pi[I]$ to be the number $\sgn(X) = (-1)^{n-k}$ where $n = |I|$ and $k$ is the number of blocks of $X$. This is the sign of any permutation of $I$ whose cycles are the blocks of $X$. Note that $\sgn(X\sqcup Y) = \sgn(X) \cdot \sgn(Y)$ for partitions $X$, $Y$ of disjoint finite sets.

\begin{corollary} The natural change of basis $e_\bullet \to h_\bullet$ is an involution of $\h$.
\end{corollary}

\begin{proof}
It follows by the preceding theorem that $\h$ has an endomorphism mapping $p_\lambda \mapsto \sgn(\lambda) \cdot p_\lambda$ for all $\lambda \in \P[I]$. This endomorphism is evidently an involution, and by Lemma \ref{basis-lem2} it coincides with the change of basis $e_\bullet \to h_\bullet$.
\end{proof}

\section{Applications to Hopf algebras}
\label{app-sect}

Putting together the results of the previous section sets up a general correspondence
\be\label{gencor-eq} \left\{ \begin{array}{c} \text{Linearized cocommutative} \\ \text{connected comonoids}  \earr\right\}\leadsto
 \left\{ \begin{array}{c} \text{Connected Hopf monoids with} \\ \text{several partially ordered bases} \earr\right\}.
 \ee
In detail, we start with a  connected comonoid 
$ \c =(\kk\Q,\Delta) $
where $\Q$ is a finite set species in which the coproduct $\Delta$ is linearized. 
Fact \ref{sl-fact} shows how to attach to this data a  connected Hopf monoid
$\h = \cS(\kk\Q,\Delta).$
When $\Delta$ is cocommutative, this Hopf monoid is strongly self-dual and strongly linearized in the basis of  $\Q$-labeled set partitions $\cS(\Q)$; Theorem \ref{order-thm} then produces a certain partial order $\{\prec\}$ on $\cS(\Q)$,
from which 
 we construct four  bases $m_\bullet$, $p_\bullet$, $e_\bullet$, and $h_\bullet$ for $\h$.

In the next subsection we review the definition of two functors
 $\cK$ and $\overline\cK$, each of which  applied
to $\h$  gives 
a  graded connected Hopf algebra with four distinguished bases.
  The first Hopf algebra $\cK(\h)$ is 
cocommutative but usually noncommutative, and maps surjectively onto the second, 
which is commutative, cocommutative, and self-dual.
We then  discuss several examples illustrating the constructions which arise from simple choices of $\c$.

\subsection{Fock functors}
\label{fock-sect}

For background on Hopf algebras, see \cite{Cartier,ReinerNotes}.
Let $\GrVec$ denote the category of $\NN$-graded $\kk$-vector spaces with graded linear maps as morphisms, and let $\Hopf(\GrVec)$ denote the category of graded Hopf algebras over $\kk$. 
Recall that a graded Hopf algebra consists of a graded vector space \[\H = \bigoplus_{n\geq 0} \H_n \in \GrVec\] with four graded linear maps $\eta : \kk \to \H$ (the unit), $\varepsilon : \H \to \kk$ (the counit), $\nabla : \H\otimes \H \to \H$ (the product), and $\Delta : \H \to \H\otimes \H$ (the coproduct). These maps must be such that $\H$ is a (co)algebra with respect to the (co)unit and (co)product, and moreover the coproduct and counit must be algebra homomorphisms.
Morphisms of graded Hopf algebras are morphisms of graded vector spaces commuting with units, counits, products, and coproducts.
 A graded Hopf algebra $\H$ is \emph{connected} if   $\H_0 \cong \kk$.
 
\begin{example}\label{hopfalg-ex1}
Let $\kk\langle x_1,\dots,x_n \rangle$ be the algebra of polynomials  over $\kk$ in $n$ noncommuting variables, with the usual grading by degree.
We view $\kk\langle x_1,\dots, x_n \rangle$ as a coalgebra whose counit has the formula $f(x) \mapsto f(0)$ and whose coproduct 
is the linear map $\Delta$ with 
$\Delta(x_i) = 1\otimes x_i + x_i \otimes 1$ and $ \Delta(fg) = \Delta(f)\Delta(g).$
These structures make $\kk\langle x_1,\dots x_n \rangle$ into a graded connected Hopf algebra.
\end{example}

\begin{example}\label{hopfalg-ex2}
Let $\kk[ x_1,\dots,x_n ]$ be the algebra of polynomials  over $\kk$ in $n$ commuting variables, with the usual grading by degree.
This is a graded connected Hopf algebra with unit, counit, product, and coproduct defined exactly as for $\kk\langle x_1,\dots, x_n \rangle$, only now with our variables commuting.
\end{example}

%

Write $\Hopf(\Sp^\circ)$ for the category of connected Hopf monoids in species.
We review here the definition of two
functors
\[ \cK : \Hopf(\Sp^\circ) \to \Hopf(\GrVec)
\qquand
\overline \cK : \Hopf(\Sp^\circ) \to \Hopf(\GrVec)
\]
which, following \cite{species}, we call  \emph{Fock functors}. 
These constructions were described originally by Stover \cite{stover} and studied in more detail by Patras, Reutenauer, and Schocker \cite{291,292,293}, but we follow the notation and exposition of \cite[Chapter 15]{species}.
We remark that \cite[Chapter 15]{species} discusses two additional Fock functors $\cK^\vee$ and $\overline \cK^\vee$, which we omit here. Our applications only involve Hopf monoids which are  finite-dimensional and self-dual, and on such Hopf monoids the values of $\cK^\vee$ and $\overline \cK^\vee$ are isomorphic to (the duals of) those of $\cK$ and $\overline \cK$; see \cite[Section 15.4.4]{species}.

We view $\cK$ and $\overline \cK$ first as functors from vector species to graded vector spaces.
Recall that we abbreviate by writing $\p[n]$ for $\p[\{1,\dots,n\}]$.
Given a vector species $\p$,  define the graded vector spaces 
\[ \cK(\p) = \bigoplus_{n\geq 0} \p[n]\qquand \overline\cK(\p) = \bigoplus_{n\geq 0} \p[n]_{S_n}\]
where $\p[n]_{S_n}$ is the vector space of $S_n$-coinvariants of $\p[n]$, that is, the quotient of $\p[n]$
by the 
subspace generated by all elements of the form $x-\p[\sigma](x)$ for permutations $\sigma \in S_n$. There is thus a  surjective homomorphism 
\be\label{surj}\cK(\p) \to \overline\cK(\p)\ee
induced by the quotient maps $\p[n] \to \p[n]_{S_n}$.
If $f : \p \to \q$ is a morphism of vector species then 
\[\cK(f) : \cK(\p) \to \cK(\q)
\qquand
\overline\cK(f) : \overline\cK(\p) \to \overline \cK(\q)
\]
are the graded linear maps whose $n$th components are respectively  $f_{[n]} : \p[n] \to \q[n]$ and the map $\p[n]_{S_n} \to \q[n]_{S_n}$ induced by $f_{[n]}$ on the quotient spaces.

\begin{remark}
If $\P$ is a basis for $\p$ 
and
$\P[n]_{S_n}$ denotes a set of representative of the distinct $S_n$-orbits in $\P[n]$,
then 
the disjoint unions $
\coprod_{n\geq 0} \P[n]
$ and $
\coprod_{n\geq 0} \P[n]_{S_n}$ are  bases for $\cK(\p)$ and $\overline \cK(\p)$.
\end{remark}

Now suppose $\h=(\p,\nabla,\Delta)$ is a connected Hopf monoid. We give $\cK(\h)$ the structure of a graded connected Hopf algebra in the following way. 
As a graded connected vector space we set $\cK(\h) = \cK(\p)$, with the  unit and counit maps  
\[ \kk \xrightarrow{\sim} \p[\varnothing] \to \cK(\p)\qquand \cK(\p) \to \p[\varnothing] \xrightarrow{\sim} \kk.\]
Here the  arrows on the far left and far right are the unit and counit of $\p$ and the  remaining arrows are the canonical inclusion and projection maps.
On $\p[m]\otimes \p[n]$ the product of $\cK(\h)$ is 
the map 
\be\label{k-prod} \p[m] \otimes \p[n] \xrightarrow{\ \ \p[\id]\otimes \p[\canon]\ \ } \p[m] \otimes \p[m+1,m+n] \xrightarrow{\ \ \nabla_{[m],[m+1,m+n]}\ \ } \p[m+n]\ee
where $\canon$ denotes the  order-preserving bijection  $[n] \to [m+1, m+n].$
To define the coproduct, fix a disjoint decomposition $[n] = S\sqcup T$ with $|S| = j$ and $|T| = k$ and define 
$\st_{S,T}
=\p[\sigma]\otimes \p[\sigma']$ where  $\sigma$ and $\sigma'$ are  the unique order-preserving bijections $S \to [j]$ and $T \to [k]$. On $\p[n]$ the coproduct of $\cK(\h)$ is then given by the sum of the maps
\be\label{k-coproduct}  
 \p[n] \xrightarrow{\ \ \Delta_{S,T}\ \ } \p[S] \otimes \p[T] \xrightarrow{\ \ \st_{S,T}\ \ } \p[j] \otimes \p[k]
\ee
 over the $2^n$ different ordered disjoint decompositions  $[n] = S\sqcup T$.
With respect to this unit, counit, product, and coproduct, $\cK(\h)$ is a graded connected Hopf algebra \cite[Theorem 15.12]{species}.  


One makes $\overline\cK(\h)$ into a Hopf algebra in a similar way. 
Again $\overline\cK(\h) = \overline\cK(\p)$ as a graded connected vector space, with the unit and counit maps defined exactly as for $\cK(\h)$.
The compositions \eqref{k-prod} and \eqref{k-coproduct} each descend to well-defined maps
\[\p[m]_{S_m} \otimes \p[n]_{S_n} \to \p[m+n]_{S_{m+n}}
\qquand
\p[n]_{S_n} \to \p[j]_{S_j} \otimes \p[k]_{S_k}.
\] 
The product and coproduct on $\overline\cK(\h)$ are then defined exactly as for $\cK(\h)$, except with these maps in place of \eqref{k-prod} and \eqref{k-coproduct}. 
With respect to this unit, counit, product, and coproduct, $\overline\cK(\h)$ is likewise a graded connected Hopf algebra \cite[Theorem 15.12]{species}.
Moreover, the map \eqref{surj} is a morphism of graded connected Hopf algebras \cite[Theorem 15.13]{species}

\begin{example}
If $C$ is an $n$-element set and $\bfE_C$ is the connected Hopf monoid  in Example \ref{E-ex},
then $ \cK(\bfE_C) \cong \kk\langle x_1,\dots,x_n\rangle$ and $ \overline\cK(\bfE_C) \cong \kk[x_1,\dots,x_n]$
and the morphism $\cK(\bfE_C) \to \overline\cK(\bfE_C)$ may be identified with the map between polynomial algebras letting the variables commute.
\end{example}

We briefly have need of the following new notation.
If $\h$ is a connected Hopf monoid then we write $\h^*$ for its \emph{dual} (see \cite[Section 8.6]{species}); likewise if $\H$ is a graded Hopf algebra then we wrote $\H^*$ for its \emph{restricted dual}, in the sense of \cite[Section 1.6]{ReinerNotes}.
If $\p$ and $\p'$ are vector species then the \emph{Hadamard product} $\p\times \p'$ is the vector species with $(\p\times \p')[S] = \p[S]\otimes \p'[S]$ for all finite sets $S$; see \cite[Definition 8.5]{species}.
If $\h= (\p,\nabla,\Delta)$ and $\h'=(\p',\nabla',\Delta')$ are connected Hopf monoids, then there is a natural connected Hopf monoid denoted $\h\times \h'$ whose underlying species is $\p\times \p'$; see \cite[Corollary 8.59]{species}.

To conclude this section, we observe some notable properties of the functors $\cK$ and $\overline \cK$. Let $\h$ be a connected Hopf monoid; the following statements then hold:
\begin{itemize}
\item $\cK(\h) \cong \overline \cK(\bfL\times \h)$ by \cite[Proposition 15.9]{species}.

\item If $\h$ is finite-dimensional then $\overline\cK(\h^*) \cong \overline\cK(\h)^*$ by \cite[Corollary 15.25]{species}.
\item  If $\h$ is cocommutative then so are $\cK(\h)$ and $\overline\cK(\h)$ by \cite[Corollary 15.27]{species}.
\item  If $\h$ is commutative then so is $\overline\cK(\h)$ by \cite[Corollary 15.27]{species}.
\end{itemize}

\subsection{Symmetric functions}\label{sym-sect}

Recall  the graded algebra  $\NCSym$ 
 of symmetric functions in  noncommuting  variables from the introduction.
This is naturally a coalgebra  with respect to the coproduct which we write informally as  $f(x_1,x_2,\dots,) \mapsto f(x_1,y_1,x_2,y_2,\dots)$;
we interpret this as a map $\NCSym \to \NCSym\otimes \NCSym$ by viewing
\[ f(x_1,y_1,x_2,y_2,\dots) \in \NCSym(x) \otimes \NCSym(y) \cong \NCSym \otimes \NCSym\]
 where $x=\{x_1,x_2,\dots\}$ and  $y=\{y_1,y_2,\dots\}$  are   sets of noncommuting variables which commute with each other.
These definitions make $\NCSym$ into a graded connected Hopf algebra which is noncommutative and cocommutative; see \cite[Theorem 4.1]{preprint}.
The more familiar Hopf algebra  of symmetric functions  $\Sym$ is defined in exactly the same way as $\NCSym$, except that $\Sym$ is located as a subalgebra
of $\kk[[x]]$ rather than $\kk\langle\langle x\rangle \rangle$, where $\kk[[x]]=\kk[[x_1,x_2,\dots]]$
denotes the algebra of formal powers series in a countably infinite set of commuting variables.
For the precise definition of $\Sym$ and its the Hopf algebra structure (which is commutative, cocommutative, and self-dual) see \cite[Section 2]{ReinerNotes} or \cite{Stan2,zel}.

As mentioned in the introduction,
both $\NCSym$ and $\Sym$ have four  ``classical bases.'' 
The classical bases of $\NCSym$, first described in \cite{Sagan1}, are indexed by  partitions $\Lambda$ of the sets $[n]$ for $n\geq 0$. One defines
 \[ m_\Lambda = \sum_{(i_1,\dots,i_n)}  x_{i_1}\cdots x_{i_n}
 \qquad\quad
 p_\Lambda = \sum_{(i_1,\dots,i_n)}  x_{i_1}\cdots x_{i_n}
  \qquad\quad
 e_\Lambda = \sum_{(i_1,\dots,i_n)}  x_{i_1}\cdots x_{i_n}
 \]
 where the first, second, and third sums are  over all $n$-tuples of positive integers respectively satisfying the following conditions:  (1) $i_j = i_k$ if and only if  $j$ and $k$ belong to the same block of $\Lambda$, (2) $i_j = i_k$ if $j$ and $k$ belong to the same block of $\Lambda$, and (3) $i_j \neq i_k$ if $j$ and $k$ belong to the same block of $\Lambda$.
The fourth basis is given by setting
 \[ h_\Lambda = \sum_{(f,L)}  
  x_{f(1)}\cdots x_{f(n)}\]
 where the sum of over all pairs $(f,L)$ consisting of a map $f : [n] \to \{1,2,3,\dots\}$ 
 and 
 a sequence $L$ of linear orderings of each of the sets $f^{-1}(i)\cap B$ 
for  $i\geq 1$ and $B \in \Lambda$. (Note that only finitely many of these intersections are nonempty.)
 The sets $\{ m_\Lambda\}$, $\{p_\Lambda\}$, $\{e_\Lambda\}$ and $\{h_\Lambda\}$
  with $\Lambda$ ranging over all partitions of $[n]$ for $n\geq 0$ 
 are then each bases of $\NCSym$.

%
  
  The classical bases of $\Sym$, by contrast, are indexed by  partitions $\lambda = (\lambda_1,\lambda_2,\dots,\lambda_\ell)$ of the integers $n \geq 0$.
  One first
  defines 
 \[ m_\lambda = \sum x_{i_1}^{\lambda_1} \cdots x_{i_\ell}^{\lambda_\ell}\]
 as the multiplicity-free sum of the distinct monomials in the orbit of $x_1^{\lambda_1}\cdots x_\ell^{\lambda_\ell}$ under the action of $S_\infty = \bigcup_{k\geq 0} S_k$.
For each letter $x \in \{p,e,h\}$ one next sets $x_\lambda = x_{\lambda_1}x_{\lambda_2}\cdots x_{\lambda_\ell}$ 
  where 
  \[ p_k = \sum_{i \geq 1} x_i^k 
  \qquand 
  e_k = \sum_{0<i_1 < \dots < i_k} x_{i_1}\cdots x_{i_k} 
  \qquand
  h_k = \sum_{0 < i_1 \leq \dots \leq i_k} x_{i_1}\cdots x_{i_k}
  \]
  for each positive integer $k$,
with  the second two sums over $k$-tuples of positive integers.
 The sets $\{ m_\lambda\}$, $\{p_\lambda\}$, $\{e_\lambda\}$ and $\{h_\lambda\}$ with $\lambda$ ranging over all partitions of nonnegative integers  are then likewise bases of $\Sym$; 
 see, e.g., \cite[Chapter 7]{Stan2}.

\begin{notation}
Given an integer partition $\lambda =(\lambda_1,\lambda_2\dots,\lambda_\ell)$ of $n$, define 
\[\lambda! = \lambda_1! \cdot \lambda_2! \cdots \lambda_\ell!
\qquand 
\lambda^! = m_1!\cdot m_2!\cdots m_n!
\] where $m_i$ is the number parts of $\lambda$ equal to $i$. Recall that the \emph{type} of a set partition is the integer partition whose parts are the sizes of the set partition's blocks.
\end{notation}

Let $\bfPi$ be the connected Hopf monoid defined in Example \ref{Pi-ex}. This Hopf monoid is strongly linearized in its natural basis of set partitions $\Pi$; indeed, we may identify 
\[ \Pi = \cS(\E)\qquand \bfPi = \cS(\kk\E,\Delta)\]
where $\E$ is the set species with $\E[S ] =\{1_\kk\}$ for all   sets $S$ and $\Delta$ is the unique coproduct on $\kk\E$ which is linearized in the basis $\E$. The results in Section \ref{bases-sect}, which give
$\bfPi$ four   distinguished bases indexed by set partitions $\Lambda$,    unify the construction of the classical bases of $\NCSym$ and $\Sym$ in the sense of the following theorem, which is closely tied to the discussion in \cite[Section 17.4.1]{species}.

\begin{theorem}\label{sym-thm}
There are Hopf algebra isomorphisms $f$ and $\overline f$
such that the diagram
\[ 
\begin{diagram}
\cK(\bfPi) && \rTo^{f} && \NCSym \\ 
\dTo &&& & \dTo \\
\overline\cK(\bfPi) && \rTo^{\overline f} && \Sym
\end{diagram}
\]
commutes (where the vertical arrows are  \eqref{ncsym2sym} and \eqref{surj}), 
and such that 
\[
\ba 
f(m_\Lambda) &= m_\Lambda\\
\overline f(m_\Lambda) &= \lambda^! \cdot m_\lambda
\ea
\qquad\qquad
\ba 
f(p_\Lambda) &= p_\Lambda\\
\overline f(p_\Lambda) &=  p_\lambda
\ea
\qquad\qquad
\ba 
f(e_\Lambda) &= e_\Lambda\\
\overline f(e_\Lambda) &= \lambda!\cdot e_\lambda
\ea
\qquad\qquad
\ba 
f(h_\Lambda) &= h_\Lambda\\
\overline f(h_\Lambda) &= \lambda! \cdot h_\lambda
\ea
\] 
whenever $\Lambda$ is a set partition of $[n]$ whose type is the integer partition $\lambda$.
\end{theorem}

\begin{proof}
There are certainly isomorphisms of graded vector spaces $f : \cK(\Pi) \to \NCSym$
and 
 $\overline f : \overline \cK(\Pi) \to \Sym$ which have $f(p_\Lambda) = p_\Lambda$
 and
 $\overline f(p_\Lambda) = p_\lambda$ for all set partitions $\Lambda$ of type $\lambda$.
 By \cite[Theorem 2.1]{Sagan1}, these maps make the diagram in the theorem commute.
The first isomorphism has $f(x_\Lambda)  =x_\Lambda$ for $x \in \{m,e,h\}$ by the change of basis formulas in \cite[Theorems 3.3 and 3.4]{Sagan1}.
The map $\overline f$, in turn, has the given formulas
on $m_\Lambda$, $e_\Lambda$, and $h_\Lambda$ by \cite[Theorem 2.1]{Sagan1}.
Since the vertical maps in the diagram  are surjective Hopf algebra morphisms, it  remains only to check that $f$ is a morphism of Hopf algebras.
This follows  by comparing formulas for the product and coproduct on $\cK(\bfPi)$ and $\NCSym$ in the $m_\bullet$ basis. For $\cK(\bfPi)$, we obtain these formulas  by combining Theorem \ref{basis-thm} with the definitions in Section \ref{fock-sect}. For $\NCSym$, the relevant formulas are \cite[Proposition 3.2]{preprint} and \cite[Eq. (6)]{preprint}.
\end{proof}

\subsection{Superclass functions}
\label{sc-sect}
We describe here a simple way of constructing  three strongly linearized Hopf monoids whose images under the Fock functor $\cK$ are the Hopf algebras of superclass functions introduced in Section \ref{sc-intro}.
For this application we need two sets of definitions.

First, let $G$ be a finite group and recall that $\E_G$ is the set species with  $\E_G[S]$ the set of maps $S \to G$. Given $g \in G$ and a map $f : S \to G$,
define $g\cdot f : S \to G$ as the map $x \mapsto gf(x)$. We write $\wt f$ for the orbit of $f \in \E_G[S]$ under this $G$-action, and define $\wt \E_G$ as the set species with 
\[ \wt \E_G [S] = \left\{ \wt f : f \in \E_G[S]\right\}
\qquand \wt \E_G[\sigma](\wt f) = \wt{f\circ \sigma^{-1}}\text{ for }f \in \E_G[S]\]
for finite sets $S$ and bijections $\sigma : S \to S'$.
The linearizations of $\E_G$ and $\wt \E_G$ are naturally cocommutative  connected comonoids, which we write as the pairs
\[ (\kk\E_G,\Delta) \qquand (\kk\wt\E_G,\Delta).\]
Here, the left and right  coproducts (abusively denoted by the same symbol $\Delta$) act on the bases $\E_G$ and $\wt \E_G$ respectively by 
$ \Delta_{S,T}(f) = f|_S \otimes f|_T
$ and $ \Delta_{S,T}(\wt f) = \wt{f|_S} \otimes \wt{f|_T}$ for maps $f : S\sqcup T \to G$, where $f|_S$ denotes the restriction of  $f$ to the domain $S$.

Next, given two connected species $\p'$ and $\p''$, define the \emph{connected sum} $\p'\csum \p''$ to be the connected species with 
$( \p'\csum\p'')[\varnothing] = \kk $ and with
 \[ ( \p'\csum\p'')[S] =\p'[S]\oplus \p''[S]
\qquand 
(\p'\csum\p'')[\sigma] = \p'[\sigma]\oplus \p''[\sigma]\] for nonempty sets $S$ and bijections $\sigma : S \to S'$.  If $\p'$ and $\p''$  have bases $\P'$ and $\P''$, then 
$\p\csum \p''$ has a basis given by the set subspecies $\P'\csum\P''$ with 
\[  (\P'\csum\P'')[\varnothing] = \{1_\kk\}\qquand (\P'\csum\P'')[S] = \P'[S]\sqcup\P''[S]\]
for nonempty sets $S$.
In turn, we define the \emph{connected sum} of two connected comonoids $(\p',\Delta')$ and $(\p'',\Delta'')$ to be the unique connected comonoid \[(\p,\Delta) = (\p',\Delta') \csum (\p'',\Delta'')\] such that $\p = \p'\csum \p''$ and  $\Delta_{S,T} = \Delta'_{S,T} \oplus \Delta''_{S,T}$ whenever $S$, $T$ are nonempty disjoint sets. (The assumption that this comonoid be connected determines $\Delta_{S,T}$ when $S$ or $T$ is empty.)
Observe that if $(\p',\Delta')$ and $(\p'',\Delta'')$ are linearized with respect to the bases $\P'$ and $\P''$,
then $(\p,\Delta)$ is linearized with respect to $\P'\csum \P''$.

Using these constructions, together with those in  Section \ref{co2hopf},
we now define three finite-dimensional, commutative, cocommutative, strongly linearized connected Hopf monoids $\bfPi_G$, $\bfPi'_G$, and $\bfPi''_G$. 
Namely, we let
\[
\ba
 \bfPi_G &= \cS(\kk\wt \E_G,\Delta) 
 \\
\bfPi'_G &= \cS(\kk\wt \E_{G\times S_2},\Delta)
\\
\bfPi''_G &= \cS\( (\kk\wt \E_{G\times S_2},\Delta) \csum (\kk\E_G,\Delta)\)
.
\ea
\]
We likewise define
\[\Pi_G = \cS(\wt \E_G)
\qquand \Pi'_G= \cS(\wt \E_{G\times S_2})
\qquand \Pi''_G = \cS(\wt \E_{G\times S_2} \csum \E_G)\] as the natural bases with respect to which these   connected Hopf monoids are strongly linearized; observe that $\Pi'_G \subset \Pi''_G$.
Each of these set species inherits a partial order from Theorem \ref{order-thm}.
To describe this order (at least in the first two cases), fix a finite set $S$ and define an action of $G$ on subsets $B $ of $ S \times G$ by 
$g\cdot B = \{ (s,gx) : (s,x ) \in B\}.$
Recall that a $G$-action on a set $X$  is \emph{free} if $g\cdot x =x$ for $g \in G$ and $x \in X$ implies $g=1$. We  have the following proposition.

\begin{proposition}
There is an order-preserving bijection between $\Pi_G[S]$ (ordered by $\prec$ as defined from the structure of $\bfPi_G$ as a linearized Hopf monoid via Theorem \ref{order-thm}) and the set of partitions of $S\times G$ on whose blocks $G$ acts freely 
(ordered by refinement).
\end{proposition}

\begin{proof}
Let $X \in \Pi_G[S]$ and recall that $X$ is a set of pairs $(B,\wt f)$ where $B \subset S$ and $\wt f$ is the $G$-orbit of a map $f : B \to G$.
Given a subset $B \subset S$ and a map $f :  B \to G$, define 
\[ \beta_{B,f} = \{ (b,f(b)) : b \in B \} \subset S\times G.\]
Now let $\Lambda_X \in \Pi[S\times G]$ be the set partition whose blocks are the sets $\beta_{B,f} $
where $B$ and $f$ range over all subsets $B \subset S$ and functions $f : B \to G$ such that $(B,\wt f) \in X$. One checks that the map $X\mapsto \Lambda_X$ is then the desired order-preserving bijection.
\end{proof}

Let $S$ be a subset of $\ZZ$, and write $-S = \{-i : i \in  S\}$. Given a partition $\Lambda$ of $S$, we define $-\Lambda$ to be the partition  $-\Lambda = \{ -B : B \in \Lambda\}$ of $-S$. We say that $\Lambda$ is \emph{symmetric} if $\Lambda = -\Lambda$; note that this can occur only if $S=-S$. We say that a pair $(i,j) \in S \times S$ with $i<j$ is an \emph{arc} of $\Lambda$ if there is a block $B \in \Lambda$ such that $i,j \in B$ and either $k\leq i$ or $j \leq k$ for each $k \in B$. In explanation of this terminology, we comment that it is sometimes convenient to visualize a  partition of $S$ as the graph with vertex set  $S$ whose edges are the set of arcs of the partition. For example,
\[\Lambda= \{ \{ 1,3,4,7\} ,\{2,6\}, \{5\} \} \quad\text{is represented by}\quad \xy<0.0cm,-0.0cm> \xymatrix@R=-0.0cm@C=.3cm{
*{\bullet} \ar @/^.7pc/ @{-} [rr]   & 
*{\bullet} \ar @/^1.0pc/ @{-} [rrrr] &
*{\bullet} \ar @/^.5pc/ @{-} [r]  &
*{\bullet} \ar @/^0.7pc/ @{-} [rrr] &
*{\bullet} &
*{\bullet}  &
*{\bullet}\\
1   & 
2 &
3  &
4 &
5 &
6 & 
7
}\endxy.\]
 A \emph{labeling} of the arcs of $\Lambda$ by a set $X$ consists of a choice of an element $x \in X$ for each arc $(i,j)$. If $X$ is finite and $\Lambda$ has $k$ blocks then 
 the number of $X$-labelings of the arcs of $\Lambda$ is $|X|^{|S|-k}$.

\begin{lemma}\label{sp-lem2} Let $n$ be a nonnegative integer, and define $[\pm n] = \{-n,\dots,-1\} \cup \{1,\dots,n\}$.
\ben
\item[(a)]  There is a bijection from $\Pi_G[n]$ to the set of partitions of $[n]$ whose arcs are $G$-labeled.

\item[(b)] There is a bijection from $\Pi'_G[n]$ to the set of symmetric partitions of $[\pm n]$ whose arcs are $G$-labeled such that if $(i,j)$ is an arc then $(-j,-i)$ is an arc with the same label, and $i+j \neq 0$.

\item[(c)] There is a bijection from $\Pi''_G[n]$ to the set of symmetric partitions of $[\pm n]$ whose 
arcs are $G$-labeled such that if $(i,j)$ is an arc then $(-j,-i)$ is an arc with the same label.
\een
\end{lemma}

\begin{proof}
Given a subset $B \subset\ZZ$ and a map $f : B \to G$, 
let $\Lambda(B,\wt f)$  be the partition of $B$ with a single block, each of whose arcs $(i,j)$ is labeled by $f(i)^{-1} f(j)$; observe that this label depends only on the orbit $\wt f$ of $f$ under the action of $G$.
Now, if we  define
\[ \varphi(X) = \coprod_{(B,\wt f) \in X} \Lambda(B,\wt f)\qquad\text{for $X \in \Pi_G[n]$}\]
then   $\varphi$ gives a bijection from $\Pi_G[n]$ to the set of partitions of $[n]$ whose arcs are $G$-labeled.

 For parts (b) and (c) we modify this notation in the following way. View $G\times S_2$ as the disjoint union $\{+g : g \in G\} \sqcup \{ -g : g \in G \}$ with multiplication given by 
 \[ (-g)(-h) = (+g)(+h) = +gh\qquand (-g)(+h) = (+g)(-h) = -gh.\]
Next  let $X \in \Pi_G''[n]$. Each block of $X$ is then either of the form (i) $(B,\wt f)$ or (ii) $(B,f)$ where $B \subset [n]$ 
 is nonempty and $f$ is either a map $B \to G\times S_2$ in the first case or a map $ B \to G$ in the second.
We  assign a $G$-labeled set partition to each such pair 
in the following way.

 For blocks of the first type, define $\Lambda_\pm(B,\wt f)$
 to be the partition of $-B \sqcup B$ with exactly two blocks
 such that if $(i,j)$ is an arc of $\sh(X)$ with $i,j \in B$ then the following properties hold:
 \begin{itemize}
 \item If $f(i)^{-1}f(j) = +g$ for some $g \in G$
 then
 $i$, $j$ belong to the same block of $\Lambda_\pm(B,\wt f)$.
 

 \item If $f(i)^{-1}f(j) = -g$ for some $g \in G$
 then
 $-i$, $j$ belong to the same block of $\Lambda_\pm(B,\wt f)$.

 
 \item If $C$ is one block of $\Lambda_\pm(B,\wt f)$ then $-C$ is the other block.

 \end{itemize}
Observe that if $(i,j)$ is an arc of $\Lambda_{\pm}(B, \wt f)$ then $(-j,-i)$ is a different arc; without loss of generality assume $|i| < |j|$ and label both of these arcs by  
 the element $g \in G$ such that $f(|i|)^{-1} f(|j|) = \pm g$.
%
%
One checks that $\Lambda_\pm(B,\wt f)$ is then a symmetric partition whose arcs are $G$-labeled such that if $(i,j)$ is an arc then $(-j,-i)$ is an arc with the same label, and $i+j \neq 0$.

Alternatively, for blocks $(B,f) \in X$ of type (ii) define $\Lambda_\pm(B,f)$ to be the partition of $-B \sqcup B$ with a single block, whose arcs of the form $(i,j)$  and $(-j,-i)$ with $j > 0$
are both labeled by $f(j)$.
If we now define 
\[ \varphi_\pm(X) = \coprod_{(B, \wt f)} \Lambda_\pm(B,\wt f) \sqcup \coprod_{(B,f)} \Lambda_\pm(B,f)\]
where the first disjoint union is over the blocks of $X$ of type (i)
and the second is over the blocks of type (ii), then $\varphi_\pm$ 
is a bijection from $\Pi''_G[n]$ to the set of symmetric partitions of $[\pm n]$ whose 
arcs are $G$-labeled such that if $(i,j)$ is an arc then $(-j,-i)$ is an arc with the same label.
Moreover, the bijection in part (b) is just the restriction of this map.
\end{proof}

Let $\FF$ denote a finite field and 
recall  from the introduction the definition of the graded vector spaces
of superclass functions on the  maximal unipotent subgroups of the special linear, even orthogonal, and symplectic groups over $\FF$, denoted
 \[ \SC(\USL_\bullet,\FF) \qquand \SC(\UO_\bullet,\FF)\qquand \SC(\USp_\bullet,\FF).\]
 Note that the second two spaces are only defined when $\FF$ has an odd number of elements.
In the following theorem, we write $\FF^\times$ for the multiplicative group of nonzero elements of $\FF$. 
 
 \begin{theorem}\label{sc-thm}
 There are isomorphisms of graded vector spaces 
 \[ \cK(\bfPi_{\FF^\times}) \cong \SC(\USL_\bullet,\FF)
 \quand
 \cK(\bfPi'_{\FF^\times}) \cong \SC(\UO_\bullet,\FF)
 \quand
 \cK(\bfPi''_{\FF^\times}) \cong \SC(\USp_\bullet,\FF)
 \]
 which endow 
 $\SC(\USL_\bullet,\FF)$, $\SC(\UO_\bullet,\FF)$, and $\SC(\USp_\bullet,\FF)$ with the structures of graded connected Hopf algebras. The Hopf algebra structures so defined on $\SC(\USL_\bullet,\FF)$ and $\SC(\UO_\bullet,\FF)$ are  isomorphic to those respectively described in \cite{AZ} and \cite{CB}.
 \end{theorem}
 
 \begin{proof}
 The first statement in the theorem is equivalent to Lemma \ref{sp-lem2}, as the sets described by that result to be in bijection with $\Pi_G[n]$, $\Pi'_G[n]$, and $\Pi''_G[n]$ are precisely the  indexing sets of the natural bases of $\SC(\USL_n,\FF)$, $\SC(\UO_{2n},\FF)$, and $\SC(\USp_{2n},\FF)$; see \cite[Section 2.1]{AZ} and \cite[Sections 3 and 4.1]{CB}.
 The second claim in the theorem
 follows by comparing:
 \begin{itemize}
 \item The product and coproduct formulas for the $m_\bullet$ basis of $\cK(\bfPi_{\FF^\times}) $ 
 with the formulas for the basis $\{\bf k_\lambda\}$  of $ \SC(\USL_\bullet,\FF)$ in \cite[Proposition 3.5 and Theorem 3.6]{AZ}.
 
\item The product and coproduct formulas for the $p_\bullet$ basis of  $\cK(\bfPi'_{\FF^\times})$ with the formulas for the basis $\{ P_\lambda\}$  of $\SC(\UO_\bullet,\FF)$ in \cite[Proposition 3.13]{CB}.
\end{itemize}
These comparisons show that the isomorphisms of graded vector spaces induced by the bijections in Lemma \ref{sp-lem2} (applied in the two cases here to identify the given pairs of bases) are isomorphisms of Hopf algebras.
  \end{proof}

We observe at this point that the connected Hopf monoid $\bfPi_G$ depends up to isomorphism only on the size of $G$. Indeed, this is true of the connected comonoid $(\kk\wt\E_G,\Delta)$,
and it therefore follows from Theorem \ref{sym-thm} that   $\SC(\USL_\bullet,\FF) \cong \NCSym$ whenever $\FF$ is a field with two elements \cite[Theorem 3.2]{AZ}.
Likewise, the following theorem derives as a corollary to the preceding result.
 
 \begin{theorem}\label{hiso-thm} Let $\FF$ and $\FF'$ be finite fields with $|\FF'|$ odd. Then  $ \SC(\USL_\bullet,\FF) $ and $ \SC(\UO_\bullet,\FF')$ are isomorphic Hopf algebras 
 if and only if $|\FF| = 2|\FF'|-1$.
 \end{theorem}
%

We close this section with some remarks about how the product and coproduct of the Hopf algebras in Theorem \ref{sc-thm} correspond to operations on superclass functions.
Whenever $i+j =n$ there are numerous injective homomorphisms 
$\USL_i(\FF) \times \USL_j(\FF) \hookrightarrow \USL_n(\FF)$
which give rise to linear maps 
 \[\SC(\USL_n,\FF) \to \SC(\USL_i,\FF) \otimes \SC(\USL_j,\FF).\]
The coproduct of $\SC(\USL_\bullet,\FF)$ on its $m_\bullet $ (and also $p_\bullet$) basis
can be naturally interpreted via these restriction maps; see the discussion in \cite[Section 2.2]{AZ}.
There are likewise 
   injective homomorphisms 
$\UO_{2i}(\FF) \times \UO_{2j}(\FF) \hookrightarrow \UO_{2n}(\FF)$
and
$\USp_{2i}(\FF) \times \USp_{2j}(\FF) \hookrightarrow \USp_{2n}(\FF)$
giving rise on superclass functions to well-defined restriction maps,  
with which one can similarly interpret the coproduct on the $m_\bullet$ bases of
$\SC(\UO_\bullet,\FF)$
and
$\SC(\USp_\bullet,\FF)$; see \cite[Section 3.2]{CB}.

Going in the opposite direction, there is a surjective homomorphism $\USL_n(\FF) \to \USL_i(\FF)\ \times \USL_j(\FF)$
 which defines an inflation map on superclass functions
 \[   \SC(\USL_i,\FF) \otimes \SC(\USL_j,\FF) \to \SC(\USL_n,\FF).\]
The product of $\SC(\USL_\bullet,\FF)$ on its $m_\bullet$ basis corresponds to this inflation map; see \cite[Section 3.1]{AZ}. 
The products of 
$\SC(\UO_\bullet,\FF)$
and
$\SC(\USp_\bullet,\FF)$ lack such an interpretation, as the corresponding towers of groups do not possess analogous surjective homomorphisms.
Rather, there is a sense in which we should view the products on the $m_\bullet$ bases of these Hopf algebras not as group theoretic operations 
but simply as artifacts of the natural orders on $\Pi_G'$ and $\Pi_G''$ at the species level; see  \cite[Theorem 4.1.9]{Me}.

\subsection{Combinatorial Hopf monoids}
\label{comb-sect}


In this final section we describe a species analogue of the notion of a \emph{combinatorial Hopf algebra} from \cite{ABS} (whose definition was given in Section \ref{intro-comb-sect}).
The scope of our discussion here is notably limited, and there is presumably much more material in \cite{ABS} with interesting generalizations to Hopf monoids in species than we have attempted to include.

\begin{definition}
A \emph{combinatorial Hopf monoid} is a pair $(\h,\zeta)$ where $\h = (\p,\nabla,\Delta)$ is a finite-dimensional connected Hopf monoid and $\zeta : (\p,\nabla) \to \bfE$ is a morphism of connected monoids.
\end{definition}

Here, the monoid structure on $\bfE$ is given by forgetting the coproduct in Example \ref{E-ex}. 
We refer to $\zeta$ as the character of $\h$.  We typically omit the subscript when referring to the $I$-components of characters, and just write $\zeta(x)$ for the image of $x \in \p[I]$ under $\zeta_I$.
Combinatorial Hopf monoids form a category in which the morphisms between two objects 
$ (\h,\zeta) $ and $ (\h',\zeta')$  are morphisms of connected Hopf monoids $f : \h \to \h'$ such that $\zeta = \zeta' \circ f$.

We may view $\cK$ and $\overline \cK$ as functors between the categories of combinatorial Hopf monoids and combinatorial Hopf algebras by merit of the following basic fact, which we state without proof:

\begin{fact}
If $(\h,\zeta)$ is a combinatorial Hopf monoid then $\cK(\h)$ and $\overline{\cK}(\h)$ are combinatorial Hopf algebras with respect to the characters 
\[\Tr\circ\cK(\zeta) \qquand \Tr\circ\overline\cK(\zeta),\] where we identify $\cK(\bfE) \cong \overline\cK(\bfE) \cong \kk[x]$ and define $\Tr$ as the map $\kk[x] \to \kk$ given by $f\mapsto f(1_\kk)$.
With respect to these conventions, if $\alpha : (\h,\zeta) \to (\h',\zeta')$ is a morphism of combinatorial Hopf monoids then $\cK(\alpha)$ and $\overline\cK(\alpha)$ are morphisms of combinatorial Hopf algebras.
\end{fact}


Define $\zeta_\Pi : \bfPi \to \bfE$ such that $\zeta_\Pi(p_\Lambda) = 1$ for all set partitions $\Lambda$. The pair $(\bfPi,\zeta_\Pi)$ is then a combinatorial Hopf monoid by Theorem \ref{basis-thm}. We note the following lemma.

\begin{lemma}\label{char-fact}
If $\Lambda$ is a set partition then $\zeta_\Pi(m_\Lambda) =\begin{cases}  1 &\text{if $\Lambda$ has at most one block} \\ 0 & \text{otherwise}.\end{cases}$  
\end{lemma}

\begin{proof}
Suppose $\Lambda \in \Pi[S]$. Let $\Lambda''$ be the unique partition of $S$ with at most one block.
Then $\zeta_\Pi(m_\Lambda) = \sum_{\Lambda' \geq \Lambda} \Mobius(\Lambda,\Lambda') \cdot \zeta_\Pi(p_{\Lambda'}) = \sum_{\Lambda \leq \Lambda' \leq \Lambda''} \Mobius(\Lambda,\Lambda').$
The last expression is equal to the desired formula by the definition of the M\"obius function.
\end{proof}

The following theorem, which  closely mirrors the analogous statement \cite[Theorem 4.3]{ABS}
for Hopf algebras, shows that $(\bfPi,\zeta_\Pi)$ is a terminal object in the full subcategory of combinatorial Hopf monoids which are cocommutative.
This result is a special case of \cite[Theorem 11.27]{species}, but for completeness (and since the more general theorem is stated using somewhat different language than what is given here) we have included a self-contained proof.

\begin{theorem}[Aguiar and Mahajan \cite{species}] \label{terminal-thm}
For any cocommutative combinatorial Hopf monoid $(\h,\zeta)$, there exists a unique morphism of combinatorial Hopf monoids
$\Psi : (\h,\zeta) \to (\bfPi,\zeta_\Pi).$
This morphism has the following formula. Given a nonempty finite set $I$ and $x \in \h[I]$,
\[\Psi_I(x) = \sum_{\Lambda \in \Pi[I]} \zeta_\Lambda(x) \cdot m_\Lambda\]
where for a set partition $\Lambda = \{B_1,\dots,B_k\} \in \Pi[S]$ we define $\zeta_\Lambda$ as the composition
\[ \h[S] \xrightarrow{\ \ \Delta_{B_1,\dots,B_k}\ \ } \h[B_1]\otimes \dots \otimes \h[B_k] \xrightarrow{\ \ \zeta_{B_1}\otimes \dots \otimes \zeta_{B_k}\ \ } \kk\otimes \cdots \otimes \kk\xrightarrow{\ \ \sim\ \ } \kk.\]
\end{theorem}

\begin{remark}
One can show that the Hopf monoid of set compositions $\mbox{\boldmath$\Sigma$}^*$ defined in \cite[Section 12.4]{species} is likewise a terminal object in the  category of all combinatorial Hopf monoids, by essentially the same argument as on the one which follows. 
\end{remark}

\begin{proof}

Abusing notation slightly,  we  write $\nabla$ and $\Delta$ for the product and coproduct of both $\h$ and $\bfPi$.
First note that the formula for $\zeta_\Lambda$ does not depend on the ordering of the blocks of $\Lambda$, and so is well-defined, since $\h$ is cocommutative. 
Since $\Delta_{B_1,\dots,B_k}$ is the identity map when $k= 1$, we have $\zeta_\Lambda(x) = \zeta(x)$ when $\Lambda$ has less than two blocks, and so it follows from the formula for $\zeta_\Pi(m_\Lambda)$ in Lemma \ref{char-fact} that $\zeta = \zeta_\Pi\circ \Psi$.

To see that $\Psi$ commutes with products and coproducts, suppose $S$ and $T$ are disjoint finite sets with $I = S\sqcup T$. 
If  $\Lambda' \in \Pi[S]$ and $\Lambda'' \in \Pi[T]$
then by definition
$\zeta_{\Lambda'\sqcup \Lambda''} = (\zeta_{\Lambda'} \otimes \zeta_{\Lambda''}) \circ \Delta_{S,T}.$
Similarly, since the product and coproduct of $\h$  are Hopf compatible and since $\zeta$ is a morphism of monoids, it follows if  $\Lambda \in \Pi[I]$ that 
$\zeta_\Lambda \circ \nabla_{S,T} =\zeta_{\Lambda|_S} \otimes \zeta_{\Lambda|_T}$ with $\Lambda|_S$  defined as in Example \ref{Pi-ex}.
From these identities and the (co)multiplication formulas for the basis $m_\bullet$ of  $\bfPi$, it is straightforward to check  that 
$\Psi$ is a morphism of combinatorial Hopf monoids. 

To show the uniqueness of $\Psi$, let $\Phi : (\h,\zeta) \to (\bfPi,\zeta_\Pi)$ be another morphism. Suppose $I$ is a set of minimal cardinality for which there exists $x \in \h[I]$ such that $\Psi_I(x) \neq \Phi_I(x)$. 
Write 
$(\Psi_I-\Phi_I)(x) = \sum_{\Lambda \in \Pi[I]} c_\Lambda  m_\Lambda$ where $c_\Lambda \in \kk.$
For any disjoint decomposition $I = S\sqcup T$ we have
\[ (\Psi_S\otimes \Psi_T -\Phi_S\otimes \Phi_T)\circ \Delta_{S,T}(x) = \Delta_{S,T}\circ (\Psi_I-\Phi_I)(x) = \sum_{(\Lambda',\Lambda'') \in \Pi[S] \times \Pi[T]} c_{\Lambda'\sqcup \Lambda''} \cdot m_{\Lambda'}\otimes m_{\Lambda''}.\]
If $S$ and $T$ are both proper subsets then $\Psi_S = \Phi_S$ and $\Psi_T = \Phi_T$ by hypothesis, in which case this expression is zero. Thus $c_\Lambda = 0$ 
whenever $\Lambda$ is a union of two partitions of nonempty sets, i.e., whenever $\Lambda$ has more than one block.
Since $(\Psi_I- \Phi_I)(x) \neq 0$, it follows that $c_\Lambda \neq 0$ for the unique partition $\Lambda \in \Pi[I]$ with at most one block. This however leads to the contradiction
\[ 0 \neq c_{\Lambda} = \zeta_\Pi \circ (\Psi_I- \Phi_I)(x) = (\zeta-\zeta)(x) = 0\]
by Lemma \ref{char-fact}.
We conclude that $\Psi = \Phi$ is the unique morphism $(\h,\zeta) \to (\bfPi,\zeta_\Pi)$.
\end{proof}

As an application, we may use the preceding theorem to construct 
 the \emph{Frobenius characteristic map} identifying symmetric functions with class functions on the symmetric group.
Towards this end,
we define  two graded vector spaces
\[\Fun(S_\bullet) = \bigoplus_{n\geq 0} \Fun(S_n)
\qquand
\Cl(S_\bullet) = \bigoplus_{n\geq 0} \Cl(S_n)\]
where $\Fun(S_n)$ is the vector space of maps $S_n \to \kk$, 
 and $\Cl(S_n)$ is the subspace of  maps $S_n \to \kk$ which are constant on the conjugacy classes of $S_n$.
When $n=0$ we identify $\Fun(S_0) = \Cl(S_0) = \kk$.

These   spaces become graded connected Hopf algebras in the following way. 
The unit and counit of both are the usual inclusions into and projections onto the subspaces of degree 0 elements.
View $S_m\times S_n$ as the subgroup of  $S_{m+n}$ 
generated by the permutations of $[m]$ and the permutations of $[m+1,m+n]$.  
The products of $\Fun(S_\bullet)$ and $\Cl(S_\bullet)$
are then given by the respective formulas 
\[ \nabla(f\otimes g) = f\times g
\qquand
\nabla(f\otimes g) = \Ind_{S_m\times S_n}^{S_{m+n}}(f \times g)
\]
where on the left $(f,g) \in \Fun(S_m)\times  \Fun(S_n)$ and on the right $(f,g) \in \Cl(S_m)\times \Cl(S_n)$.
Here, we interpret $f\times g $ as a map $S_{m+n} \to \kk$ by setting its value to be identically zero on the complement of the subgroup $S_m \times S_n$.
%
\begin{remark}
Recall that if $G$ is a finite group with a subgroup $H$ and $f$ is a function $H \to \kk$, then the induced function $\Ind_H^G(f) : G \to \kk$ is given by the formula
\[ \Ind_H^G(f)(x) = \frac{1}{|H|} \sum_{\substack{g \in G \\ gxg^{-1} \in H}} f(gxg^{-1})\qquad\text{for }x \in G.\]
This function is always constant on the conjugacy classes of $G$.
\end{remark}
The restriction of a function $f : S_{m+n} \to \kk$ to the subgroup $S_m \times S_n$ defines an element of $\Fun(S_m)\otimes \Fun(S_n)$ which we denote by $\Res^{S_{m+n}}_{S_m\times S_n}(f)$. 
The coproduct on both $\Fun(S_\bullet)$ and $\Cl(S_\bullet)$ is then given by the formula
\[ \Delta(f) = \sum_{k=0}^n \Res_{S_k\times S_{n-k}}^{S_n}(f)\]
where $f \in \Fun(S_n)$ or $f \in \Cl(S_n)$ as appropriate.
This is a well-defined coproduct on $\Cl(S_\bullet)$ since 
 if $f$ is constant on the conjugacy classes of $S_{m+n}$ then its restriction is constant on the conjugacy classes of $S_m \times S_{n}$. 

One can check directly that these definitions make $\Fun(S_\bullet)$ into a graded connected Hopf algebra. The proof of the same statement for $\Cl(S_\bullet)$ 
is more involved; see \cite[Corollary 4.9]{ReinerNotes}.
On the other hand,
the isomorphisms 
described in the following theorem
may be taken as a proof that these graded vector spaces are in fact Hopf algebras.
Similarly, consider the  morphism of graded vector spaces 
\be\label{cl-surj} \Fun(S_\bullet) \to \Cl(S_\bullet)\ee
which sends a function $f : S_n \to \kk$ to the class function $\wt f : S_n \to \kk$ defined by  
\[ \wt f(x) = \sum_{\sigma \in S_n} f(\sigma x \sigma^{-1})\qquad\text{for }x \in S_n.\]
It is an exercise to check that this graded linear map is morphism of Hopf algebras; alternatively, this fact can be deduced from our next theorem.

\begin{notation}
Given an integer partition $\lambda = (\lambda_1,\lambda_2,\dots,\lambda_\ell)$ of $n$, recall the definitions of $\lambda! $ and $\lambda^!$ from 
before Theorem \ref{sym-thm}, and let $z_\lambda = \lambda^!\cdot \lambda_1\cdot \lambda_2 \cdots \lambda_\ell$. Then $z_\lambda$ is the size of the centralizer in $S_n$ of any permutation with cycle type $\lambda$.
Also, write $S_\lambda$ for the Young subgroup 
\[S_{\lambda_1}\times S_{\lambda_2}\times \dots \times S_{\lambda_\ell} \subset S_n.\]
If $\sigma \in S_n$ then we let $1_\sigma : S_n \to \kk$ denote the characteristic function of the set $\{\sigma\}$.
Likewise, if $\lambda$ is an integer partition of $n$ then we let $1_\lambda : S_n \to \kk$ denote the characteristic function of the set of permutations in $S_n$ with cycle type $\lambda$.
Finally, we write $\One$ and $\sgn$ respectively for the trivial homomorphism $S_n \to \{1\}$ and the sign homomorphism $S_n \to \{\pm 1\}$.
\end{notation}

Recall the definition of the connected Hopf monoid $\bfSigma$ from Example \ref{S-ex}.
This Hopf monoid is commutative, cocommutative, and strongly linearized in its natural basis $\fk S$.
The order $\{ \prec\}$ on $\fk S$ defined from the product and coproduct of $\bfSigma$ by Theorem \ref{order-thm} has the following description, which we recall from \cite[Example 4.1.6]{Me}.
If $I$ is a finite set and $\lambda,\lambda' \in \fk S[I]$ are permutations 
then  $\lambda \preceq \lambda'$ if and only the cycles of $\lambda'$ are each shuffles of some number of cycles of $\lambda$. Here, a cycle $c$ is a \emph{shuffle} of two cycles $a$ and $b$ if we can write $c=(c_1,\dots, c_n)$ and there are indices $1\leq i_1 < \dots < i_k \leq n$ such that $a = (c_{i_1},\dots, c_{i_k})$ and $b$ is the cycle given by deleting $c_{i_1},\dots,c_{i_k}$ from $c$. For example, there are six  cycles which are shuffles of $(1,2)$ and $(3,4)$:
\[ (1,2,3,4),\quad (1,2,4,3),\quad (1,3,2,4),\quad (1,4,2,3),\quad (1,3,4,2),\quad (1,4,3,2).\]
In turn, a cycle $c$ is a shuffle of $n$ cycles $a_1,\dots,a_n$ if $c$ is a shuffle of $a$ and $b$ where  $a=a_1$ and $b$ is some shuffle of $a_2,\dots,a_n$. (When $n=1$ this occurs precisely when $c=a_1$.)

\begin{remark}
Using  the formulas in Theorem \ref{basis-thm}, one checks that $\cK(\bfSigma)$ is isomorphic to the graded connected Hopf algebra  $\mathfrak{S}\textbf{Sym}$ described in \cite[Proposition 3.3]{HNT}, via the map identifying the $\{p_\sigma\}$ basis of $\cK(\bfSigma)$ with the basis of $\mathfrak{S}\textbf{Sym}$ denoted $\{\textbf{S}^\sigma\}$ in \cite{HNT}.
By \cite[Corollary 3.5]{HNT}, both Hopf algebras are isomorphic  to the \emph{Grossman-Larson Hopf algebra of heap-ordered trees} studied in \cite{GL}.
\end{remark}

The distinguished bases 
of $\bfSigma$ (defined via Section \ref{bases-sect} with respect to $\fk S$) are now related to the Hopf algebras $\Fun(S_\bullet)$ and $\Cl(S_\bullet)$ by the following theorem.

\begin{theorem}\label{cl-thm}
There are Hopf algebra isomorphisms $f$ and $\overline f$
such that the diagram
\[ 
\begin{diagram}
\cK(\bfSigma) && \rTo^{f} && \Fun(S_\bullet) \\ 
\dTo &&& & \dTo \\
\overline\cK(\bfSigma) && \rTo^{\overline f} && \Cl(S_\bullet)
\end{diagram}
\]
commutes (where the vertical arrows are \eqref{surj} and \eqref{cl-surj}), 
and such that  
\[
f(p_\sigma) =1_\sigma
\qquad\quad
\overline f(p_\sigma) = z_\lambda \cdot 1_\lambda
\qquad\quad
\overline f(e_\sigma) = \lambda! \cdot \Ind_{S_\lambda}^{S_n} (\sgn)
\qquad\quad
\overline f(h_\sigma) =  \lambda! \cdot \Ind_{S_\lambda}^{S_n} (\One)
\] 
whenever $\sigma \in S_n$ is a permutation with cycle type $\lambda$.
\end{theorem}

\begin{proof}
There is certainly an isomorphism of graded vector spaces $f : \cK(\bfSigma) \to \Fun(S_\bullet)$ with $f(p_\sigma) = 1_\sigma$ for all permutations $\sigma \in S_n$ for $n\geq 0$. In view of the formulas for the product and coproduct of $\bfSigma$ on $p_\bullet$ in Theorem \ref{basis-thm}, it is clear that this map is an isomorphism of Hopf algebras.
There is likewise an isomorphism of graded vector spaces $\overline f : \overline\cK(\bfSigma) \to \Cl(S_\bullet)$ with $\overline f(p_\sigma) = z_\lambda\cdot 1_\lambda$
whenever $\sigma \in S_n$ is a permutation with cycle type $\lambda$.
One checks that if $\sigma \in S_n$ is a permutation with cycle type $\lambda$ 
then the image of $1_\sigma$ under the map \eqref{cl-surj} is 
precisely $z_\lambda \cdot 1_\lambda$.
Hence the maps $f$ and $\overline f$ make the diagram in the theorem commute.

It remains to check that $\overline f$  is a morphism of Hopf algebras, and that the images of $e_\sigma$ and $h_\sigma$ under this map are as described.
To this end, let $\sigma' \in S_m$ and $\sigma'' \in S_n$ be permutations with cycle types $\alpha$ and $\beta$.
Define $\sigma = \sigma' \times \sigma'' \in S_m\times S_n \subset S_{m+n}$
and let $\lambda$ denote its cycle type.
It is straightforward to compute
\[ \Ind_{S_m\times S_n}^{S_{m+n}}( 1_\alpha \otimes 1_\beta) = \frac{1}{m!n!} \cdot \frac{m!}{z_\alpha} \cdot \frac{n!}{z_\beta}\cdot z_\lambda \cdot 1_\lambda = \frac{z_\lambda}{z_\alpha z_\beta} \cdot 1_\lambda.\]
Since the product of $p_{\sigma'}$ and $p_{\sigma''}$ in $\overline \cK(\bfSigma)$ is $p_{\sigma}$, it follows from this identity that $\overline f$ is an algebra homomorphism.
Now suppose $\sigma \in S_n$ is a permutation whose cycle type is the partition $\lambda = (n)$ with one part. Both $p_\sigma$ and $\overline f(p_\sigma)$ are then primitive in $\overline \cK(\bfSigma)$ and $\Cl(S_\bullet)$ respectively, so since  basis elements $p_\sigma$ of this form generate $\overline \cK(\bfSigma)$ as an algebra, it follows that $\overline f$ is also a coalgebra homomorphism. Thus $\overline f$ is an isomorphism of Hopf algebras as desired.

Observe from Theorem \ref{basis-thm} that $\nabla(e_{\sigma'} \otimes e_{\sigma''}) = e_{\sigma}$
and 
$
\nabla(h_{\sigma'} \otimes h_{\sigma''}) = h_{\sigma}$
whenever $\sigma'$ and $\sigma''$ are permutations and  $\sigma = \sigma'\times \sigma''$, where both products are computed in $\overline \cK(\bfSigma)$.
Since $\overline f$ is an algebra homomorphism, to prove the given formulas for $\overline f (e_\sigma)$ and $\overline f (h_\sigma)$ it thus suffices to prove these formulas in the case when $\sigma$ is a permutation with only one cycle.
Assume $\sigma \in S_n$ has this form,
and for each set partition $\Lambda \in \Pi[n]$ let $\sigma_\Lambda$  be a permutation with $\sigma_\Lambda \preceq \sigma$ (in the order $\prec$ on $\fk S$ defined by Theorem \ref{order-thm}) whose cycles are the blocks of $\Lambda$. By Theorem \ref{posetisom-thm} there is exactly one permutation $\sigma_\Lambda$ with these properties and by Lemma \ref{basis-lem2} it holds that 
\be\label{||} e_\sigma = \sum_{\Lambda \in \Pi[n]} \sgn(\sigma_\Lambda) \cdot \|\Lambda\| \cdot p_{\sigma_\Lambda}, \qquad\text{where }\|\Lambda\| = \prod_{B \in \Lambda} (|B|-1)!.\ee
Since $\|\Lambda\|$ is the number permutations whose cycles are the blocks of $\Lambda$,
it follow that the sum $\sum_\Lambda \|\Lambda\|$ over all set partitions with type $\lambda$ is exactly $\frac{n!}{z_\lambda}$, the size of the conjugacy class of permutations with cycle type $\lambda$. Thus when $\sigma$ is a permutation with one cycle
\[ \overline f(e_\sigma) = \sum_{\Lambda \in \Pi[n]} \sgn(\sigma_\Lambda) \cdot \|\Lambda\| \cdot \overline f(p_{\sigma_\Lambda})
=
\sum_{\lambda \vdash n}   \sgn(\lambda) \cdot \tfrac{n!}{z_\lambda} \cdot z_\lambda \cdot 1_\lambda
= n! \cdot \sgn
\]
and it follows by a similar argument that $\overline f(h_\sigma) = n! \cdot \One$, which is what we needed to show.
\end{proof}

We now  combine in the following corollary Theorems \ref{sym-thm}, \ref{terminal-thm}, and \ref{cl-thm}
to construct an isomorphism between $ \Cl(S_\bullet) $ and $ \Sym$; this isomorphism is precisely the well-known Frobenius characteristic map.

\begin{corollary}\label{frob-cor}
There is an isomorphism $ \Cl(S_\bullet) \xrightarrow{\sim} \Sym$ of graded connected Hopf algebras
 such that  if $\lambda$ is a partition of a nonnegative integer $n$ then
\[   z_\lambda\cdot 1_\lambda \mapsto p_\lambda \qquand  \Ind_{S_\lambda}^{S_n}(\One)\mapsto h_\lambda \qquand  \Ind_{S_\lambda}^{S_n}(\sgn) \mapsto e_\lambda.\]
\end{corollary}

\begin{remark}
The irreducible characters of $S_n$ for $n\geq 0$ provide another basis for $\Cl(S_\bullet)$, and the image of this basis under the Frobenius characteristic map is  the basis of Schur functions $\{s_\lambda\}$ in $\Sym$; see  \cite[Theorem 4.10]{ReinerNotes}.
We note that the paper \cite{ER} describes the preimage of the $\{m_\lambda\}$ basis in $\Cl(S_\bullet)$ under the Frobenius map; this is considerably more complicated than for our other bases.
\end{remark}

\begin{proof}
Define  $\zeta_{\fk S} : \bfSigma \to \bfE$ as the morphism with  $\zeta_{\fk S}(p_\sigma) = 1$ for all permutations $\sigma$.
Then $(\bfSigma,\zeta_{\fk S})$ is a combinatorial Hopf monoid, and 
 the unique morphism
\[\Psi : (\bfSigma, \zeta_{\fk S}) \to (\bfPi,\zeta_\Pi)\] given by Theorem \ref{terminal-thm}
is such that
if
$\sigma \in \fk S[I]$ is a permutation then $\Psi_I(p_\sigma) = \sum_{\Gamma} m_\Gamma$ where 
the sum is over all set partitions $\Gamma  \in \Pi[I]$ such that each cycle of $\sigma$ is a subset of a block of $\Gamma$.
This is precisely the element $p_\Lambda \in \bfPi[I]$ where $\Lambda=\sh(\sigma)$ is the set partition whose blocks are the cycles of $\sigma$.
The morphism $\Psi$ is thus   surjective, so  
\[\overline\cK(\Psi) : \overline \cK(\bfSigma) \to \overline \cK(\bfPi)\]
 is a surjective morphism of graded connected Hopf algebras. As the quotient spaces of $S_n$-coinvariants in $\bfSigma[n]$ and $\bfPi[n]$ always have the same dimension, 
$\overline\cK(\Psi)$
is  in fact an isomorphism.
The statement of the corollary now follows  by conjugating $\overline\cK(\Psi)$ with the isomorphisms 
$\overline\cK(\bfPi) \cong \Sym$ and $\overline\cK(\bfSigma) \cong \Cl(S_\bullet)$ in 
 Theorems \ref{sym-thm} and \ref{cl-thm}. 
\end{proof}

These results lend themselves to a very natural construction of the lifting map $ \Sym \to \NCSym$ considered in \cite[Section 4]{Sagan1}.
There is an obvious lifting map $\Cl(S_\bullet) \to \Fun(S_\bullet)$, namely, the map given by the natural inclusion of the space of class functions in the space of all functions on $S_n$ for each $n\geq 0$.
Write $\rho : \NCSym \to \Sym$ for the projection \eqref{ncsym2sym} and 
define $\tilde \rho$ as the graded linear map making the  diagram 
\[ 
\begin{diagram}
 \Fun(S_\bullet) && \rTo^{\sim} && \cK(\bfSigma) && \rTo^{\cK(\Psi)} && \cK(\bfPi) && \rTo^{\sim}&& \NCSym \\ 
\uTo &&& & &&&&&&&& \uTo_{\tilde\rho}  \\
  \Cl(S_\bullet) && \lTo^{\sim} && \overline\cK(\bfSigma) && \lTo^{\overline\cK(\Psi)^{-1}} && \overline \cK(\bfPi) && \lTo^{\sim} && \Sym
\end{diagram}
\]
commute, where the horizontal isomorphisms are the (inverses of the) ones in Theorems \ref{sym-thm} and \ref{cl-thm}, and $\Psi$ is the unique morphism of   combinatorial Hopf monoids $(\bfSigma, \zeta_{\fk S}) \to (\bfPi,\zeta_\Pi)$.

\begin{proposition}
The map $\tilde \rho  $ thus defined is such that 
$\rho\circ \tilde \rho$ acts as multiplication by $n!$ on the subspace of degree $n$ elements in $\Sym$.
Moreover, if $\lambda $ is an integer partition
  of $n\geq 0$,
then the images of $m_\lambda$, $p_\lambda$, $e_\lambda$,  $h_\lambda$ under $\tilde \rho$ are 
\be\label{expr}  \lambda! \cdot \sum_\Lambda m_\Lambda
\qquand
 \lambda!\cdot  \lambda^!\cdot \sum_\Lambda p_\Lambda
 \qquand
 \lambda^! \cdot \sum_\Lambda e_\Lambda
  \qquand
 \lambda^! \cdot \sum_\Lambda h_\Lambda
   \ee
respectively, where each sum is over all set partitions $\Lambda$ of $[n]$ of type $\lambda$.
\end{proposition}

\begin{proof}
Using  Corollary \ref{frob-cor} and the diagram which defines $\tilde \rho$, it is straightforward
to compute
\[ \tilde \rho(p_\lambda) =  \sum_\Lambda z_\lambda \cdot  \|\Lambda\|\cdot  p_\Lambda =  \lambda!\cdot  \lambda^!\cdot \sum_\Lambda p_\Lambda\] where   
$\|\Lambda\|$ is defined as in \eqref{||} and the
sums are over all set partitions $\Lambda$ of $[n]$ of type $\lambda$.
This formula shows that 
the subspace of degree $n$ elements  in  the image of $\tilde \rho$ is 
the same as the image of the subspace of $S_n$-invariant elements in $\bfPi[n]$ under the isomorphism $f: \cK(\bfPi) \to \NCSym$ in Theorem \ref{sym-thm}.
All of the expressions in \eqref{expr} clearly belong to the latter image,
and it follows from Theorem \ref{sym-thm}  that applying $\rho$ to these expressions gives
\[n!  \cdot m_\lambda
\qquand
n!  \cdot p_\lambda
\qquand
n!  \cdot e_\lambda
\qquand
n!  \cdot h_\lambda\]
respectively (since there are $\frac{n!}{\lambda^!\lambda!}$   partitions $\Lambda$ of $[n]$ with  type $\lambda$; see \cite[Eq. (4)]{Sagan1}).
Since we already computed $\tilde \rho(p_\lambda)$, this shows that $\rho\circ \tilde \rho$ acts as multiplication by $n!$ on the subspace of degree $n$ elements in $\Sym$.
This fact, in turn, confirms the given formulas for $\tilde\rho$ on $m_\lambda$, $e_\lambda$, and $h_\lambda$.
\end{proof}

The proposition shows that $\tilde \rho$ is  a rescaled version of the lifting map in \cite[Section 4]{Sagan1}, which is given by the formula
 $m_\lambda \mapsto \frac{\lambda!}{n!} \cdot \sum_\Lambda m_\Lambda$ for partitions $\lambda$ of $n$.
If $\{ s_\lambda\}$ denotes the basis of Schur functions in $\Sym$, then the images \[S_\lambda \omdef = \tilde \rho(s_\lambda)\] under our lifting map are precisely the ``Schur functions'' of $\NCSym$ defined in \cite[Section 6]{Sagan1}; see \cite[Theorem 6.2(iv)]{Sagan1}.
Remarkably, Rosas and Sagan also gave an explicit combinatorial definition of the elements $S_\lambda$ in terms of MacMahon symmetric functions.
While the elements $\{S_\lambda\}$ are linearly independent, they do not form a basis, and it
remains an open problem to describe a meaningful basis for $\NCSym$ which projects via $\rho$ to the basis of Schur functions in $\Sym$.

\end{document}